\documentclass[11pt,notitlepage]{article}
\usepackage{amssymb, amsmath, amsthm, cancel}

\usepackage{amssymb,amsmath,amsfonts,mathtools}
\usepackage{color}
\usepackage[colorlinks,linkcolor=blue,citecolor=red]{hyperref}
\usepackage[nohug]{diagrams}
\usepackage{mathrsfs}

\usepackage{graphicx}
\usepackage{caption}
\usepackage{subcaption}
\usepackage{enumitem}
\usepackage{tabularx}
\usepackage{rotating}

\newtheorem{theorem}{Theorem}
\newtheorem{lemma}[theorem]{Lemma}
\newtheorem{proposition}[theorem]{Proposition}

\numberwithin{equation}{section}

\newcommand{\R}{\mathbb{R}}  
\newcommand{\A}{\mathcal{A}}  

\newcommand{\G}{\mathcal{G}}  
\newcommand{\fG}{\mathfrak{G}}  

\newcommand{\aG}{\vec{\G}} 
\newcommand{\faG}{\vec{\fG}} 
\newcommand{\cA}{\mathcal{A}}  
\newcommand{\sL}{\mathscr{L}} 
\newcommand{\sA}{\mathscr{A}} 
\newcommand{\sT}{\mathscr{T}}  
\newcommand{\sTpol}{\mathscr{T}_{\text{pol}}}
\newcommand{\sTlin}{\mathscr{T}_{\text{lin}}}
\newcommand{\Z}{\mathcal{Z}}
\newcommand{\gm}{\mathsf{g}}  

\newcommand{\X}{\mathcal{X}}  
\renewcommand{\L}{\mathcal{L}}  
\newcommand{\sfT}{\mathsf{T}}  
\newcommand{\tid}{t_{\text{id}}}  
\newcommand{\wt}{\widetilde}  
\newcommand{\TxX}{\sfT_x\X}  
\newcommand{\cTxX}{\sfT^*_x\X}  
\newcommand{\E}{\mathcal{E}}  

\begin{document}

\title{On Accelerated Methods in Optimization}
\author{Andre Wibisono \\ \texttt{wibisono@berkeley.edu}
\and
Ashia C. Wilson \\ \texttt{ashia@berkeley.edu}}

\maketitle

\begin{abstract}
In convex optimization, there is an {\em acceleration} phenomenon in which we can boost the convergence rate of certain gradient-based algorithms.
We can observe this phenomenon in Nesterov's accelerated gradient descent, accelerated mirror descent, and accelerated cubic-regularized Newton's method, among others.
In this paper, we show that the family of higher-order gradient methods in discrete time (generalizing gradient descent) corresponds to a family of first-order rescaled gradient flows in continuous time. On the other hand, the family of {\em accelerated} higher-order gradient methods (generalizing accelerated mirror descent) corresponds to a family of second-order differential equations in continuous time, each of which is the Euler-Lagrange equation of a family of Lagrangian functionals.
We also study the exponential variant of the Nesterov Lagrangian, which corresponds to a generalization of Nesterov's restart scheme and achieves a linear rate of convergence in discrete time.
Finally, we show that the family of Lagrangians is closed under time dilation (an orbit under the action of speeding up time), which demonstrates the universality of this Lagrangian view of acceleration in optimization. 
\end{abstract}

\section{Introduction}
\label{Sec:Intro}
In convex optimization, many discrete-time algorithms can be interpreted as discretizing a continuous-time curve converging to the optimal solution $f^*$ of the optimization problem:
\begin{align*}
 \min_{x\in \X} f(x).
\end{align*}
For example, the classical {\em gradient descent algorithm} in discrete time ($k \in \{0,1,2,\dots\})$
\begin{align}\label{Eq:GradDesc}
x_{k+1} = x_k - \epsilon \nabla f(x_k)
\end{align}
can be viewed as the algorithm obtained by applying the backward-Euler method to discretize {\em gradient flow} ($t\geq 0$)
\begin{align}\label{Eq:GradFlow}
\dot X_t = -\nabla f(X_t).
\end{align}
Many methods, including~\eqref{Eq:GradDesc} and~\eqref{Eq:GradFlow} above, can be interpreted as minimizing a regularized approximation of the objective function $f$.
Indeed, gradient descent can be written as:
\begin{align}\label{Eq:GradDescOpt}
x_{k+1} &= x_k + v_k \notag\\
v_k  &= \arg\min_{v } \left\{ f(x_k) + \langle \nabla f(x_k), v \rangle + \frac{1}{\epsilon} \cdot \frac{1}{2} \|v\|^2 \right\}
\end{align}
whereas gradient flow can be written as:
\begin{align}\label{Eq:GradFlowOpt}
\dot X_t = \arg\min_{v} \left\{ f(X_t) + \langle \nabla f(X_t), v \rangle + \frac{1}{2} \|v\|^2 \right\}.
\end{align}
Moreover, \eqref{Eq:GradDescOpt} and \eqref{Eq:GradFlowOpt} have matching convergence rates; gradient descent has convergence rate
\begin{align}\label{Eq:GradDescRate}
f(x_k) - f^* \leq  O\Big(\frac{1}{\epsilon k}\Big)
\end{align}
when $f$ is smooth (has $(1/\epsilon)$-Lipschitz gradients) and convex, where the $O(\frac{1}{\epsilon k})$ term in the bound above is with respect to $k \to \infty$, with $\epsilon$ fixed (more precisely, $f(x_k)-f^* \le \frac{C}{\epsilon k}$ for all sufficiently large $k$, for some constant $C > 0$).
Similarly, gradient flow has convergence rate
\begin{align}\label{Eq:GradFlowRate}
f(X_t) - f^* \leq O \left(\frac{1}{t}\right)
\end{align}
when $f$ is convex, without requiring smoothness, and the $O(\frac{1}{t})$ term above is with respect to $t \to \infty$. Note that the backward-Euler method discretizes the curve using the identification $x_k = X_t$, $x_{k+1} = X_{t + \delta} \approx X_t + \delta \dot X_t = x_k + v_k$, with the time-step $\delta$ set equal to the step size $\epsilon$ of the discrete time algorithm (equivalently, with time scaling $t = \epsilon k$).

\subsection{Summary of Results}
\label{Sec:IntroTime}
The link between discrete-time algorithms and continuous-time curves, and their matching properties (i.e. convergence rates) extends far beyond gradient descent~\eqref{Eq:GradDesc} and gradient flow~\eqref{Eq:GradFlow}. We begin (Section~\ref{Sec:HigherGrad}) by studying \emph{higher-order gradient} algorithms $\G_p$ ($p\geq 2$): 
\begin{align}\label{Eq:HigherOrder1}
x_{k+1} &= x_k + v_k \notag\\
v_k &= \arg \min_{v} \,\left\{ f_{p-1}(v;x_k) + \frac{1}{\epsilon} \cdot \frac{1}{p} \|v\|^p \right\},
\end{align}
where $f_{p-1}(v;x)$ is the $(p-1)$-st Taylor approximation of $f(x+v)$ centered at $x$:
\begin{align}\label{Eq:HigherTaylor}
f_{p-1}(v,x_k) \,=\, \sum_{i=0}^{p-1} \frac{1}{i!}f(x_k) v^i \,=\, f(x_k) + \langle \nabla f(x_k),v\rangle + \dots +\frac{1}{(p-1)!}\nabla^{p-1}f(x_k)[v,\dots,v].
\end{align} 
The $p$-th order gradient method $\G_p$, with the ansatz $x_k = X_t$, $x_{k+1} = X_{t + \delta} \approx X_t + \delta \dot X_t = x_k + v_k$, and time-step $\delta = \epsilon^{\frac{1}{p-1}}$ (equivalently, with time scaling $t = \epsilon^{\frac{1}{p-1}} k)$, discretizes a $p$-th order \emph{rescaled gradient flow}: 
\begin{equation}\label{Eq:GradFlow1}
\dot X_t = \arg\min_{v} \left\{f(X_t)+ \langle \nabla f(X_t), v\rangle + \frac{1}{p} \|v\|^p \right\}.
\end{equation}
Note \eqref{Eq:GradFlow1} can also be written as:
\begin{align}\label{Eq:RescGradFlow1}
\dot X_t = \frac{\nabla f(X_t)}{||\nabla f(X_t)||_*^{\frac{p-2}{p-1}}}.
\end{align}
 Furthermore, \eqref{Eq:HigherOrder1} and \eqref{Eq:GradFlow1} have matching convergence rates; the $p$-th order gradient method has convergence rate:
 \begin{equation*}
 f(x_k) - f^* \leq O\Big(\frac{1}{\epsilon k^{p-1}}\Big)
 \end{equation*}
 when $\nabla^{(p-1)}f$ is $\frac{(p-1)!}{\epsilon}$-smooth and $f$ is convex, and the rescaled gradient flow has convergence rate:
  \begin{equation*}
 f(X_t) - f^* \leq O\Big(\frac{1}{t^{p-1}}\Big)
 \end{equation*}
 when $f$ is convex. 
 
 In Section~\ref{Sec:AccHigherGrad}, we present an algorithm which generalizes accelerated gradient descent \cite{Nesterov04} and the accelerated Newton method \cite{Nesterov08}, and accelerates the family of higher-order gradient algorithms \eqref{Eq:HigherOrder1}. Building on the work of Su, Boyd, and Candes~\cite{SuBoydCandes14} (for the $p=2$ Euclidean case), in Section~\ref{Sec:DiscToCts} we show that the $p$-th order accelerated gradient method $\aG_p$ discretizes a second-order differential equation we call \emph{Nesterov flow}:
 \begin{equation*}
\ddot X_t + \frac{p+1}{t} \dot X_t + Cp^2t^{p-2} \nabla^2 h\left(X_t+\frac{t}{p}\dot X_t\right)^{-1} \nabla f(X_t) = 0
 \end{equation*} 
 under the time step $\delta = \epsilon^{\frac{1}{p}}$ (or time scaling $t^p = \epsilon k^p$).
 Moreover, the $p$-th order accelerated gradient algorithm $\aG_p$ and its corresponding Nesterov flow have matching convergence rates;  the $p$-th order accelerated gradient method has convergence rate:
 \begin{equation*}
 f(x_k) - f^* \leq O\Big(\frac{1}{\epsilon k^{p}}\Big)
 \end{equation*}
 when $\nabla^{(p-1)}f$ is $\frac{(p-1)!}{\epsilon}$-smooth and $f$ is convex, and the corresponding Nesterov flow has convergence rate (Section~\ref{Sec:NestFlowRate}):
  \begin{equation*}
 f(X_t) - f^* \leq O\Big(\frac{1}{t^{p}}\Big)
 \end{equation*}
 when $f$ is convex.  Note that the family of Nesterov flows are \emph{second-order} differential equations in time and the rescaled gradient flows are \emph{first-order} differential equations in time. 
 
 In Section~\ref{Sec:Lag}, we show that the Nesterov flows are a subfamily of the \emph{Bregman flows}: 
\begin{equation*}
\ddot X_t + \dot \gamma_t \, \dot X_t + e^{\beta_t} \, \nabla^2 h\left(X_t + e^{-\alpha_t} \dot X_t\right)^{-1} \nabla f(X_t) = 0
\end{equation*}
where $\alpha_t = -\log t + \log p$, $\beta_t = (p-2)\log t + 2\log p + \log C$, and $\gamma_t =\, (p+1) \log t -\log p$. Under an \emph{ideal scaling} relationship between $\alpha_t, \beta_t,\gamma_t$ (satisfied by the Nesterov flows), each Bregman flow satisfies the Euler-Lagrange equation of a \emph{Bregman Lagrangian} functional:
\begin{align}\label{Eq:BregLag1}
\L_{\alpha,\beta,\gamma}(x,v,t) \,=\, e^{\gamma_t}\left(e^{2\alpha_t} D_h\left(x+e^{-\alpha_t} v, x\right) - e^{\beta_t} f(x) \right).
\end{align}
Therefore, the family of Nesterov flows~\eqref{Eq:NestFlow} can be interpreted as optimal curves under the {\em principle of least action}, which posits that curves evolve so as to minimize a quantity known as an \emph{action}, defined as the time integral of a Lagrangian functional $\L(X,\dot X,t)$.

In Section~\ref{Sec:expLag}, we introduce \emph{exp-Nesterov flows}, another subfamily of the Bregman flows that satisfies the ideal scaling (where $\alpha_t \,=\, \log c,\, \beta_t \,=\, ct + 2\log c,\, \gamma_t \,=\, ct-\log c $). The exp-Nesterov flows have an improved convergence rate:
\[ f(X_t) - f^* \leq O\Big(\frac{1}{e^{ct}}\Big). \]
In Section~\ref{Sec:NestRest}, we show how to discretize the exp-Nesterov flow, and with the additional assumption of uniform convexity, obtain a discrete-time algorithm with matching linear rate. The algorithm presented generalizes the restart scheme of Nesterov~\cite{Nesterov08}, and is optimal (attains  the lower bound~\cite[Section~2.2.1]{Nesterov04}) when $f$ is both smooth and strongly convex (i.e. $p=2$).
 
Finally, in Section~\ref{Sec:Time} we demonstrate how \emph{time} can be used as an organizing tool to understand the various algorithms presented in this paper. Indeed, in continuous time optimization, if we start with a curve $X_t$ with convergence rate $f(X_t) - f^* \le O(e^{\rho_t})$, we can simply consider the sped-up curve $Y_t = X_{\tau(t)}$, where $\tau \colon \R_+ \rightarrow \R_+$ is a monotonically increasing function. This new curve $Y_t$ has faster convergence rate $f(Y_t) - f^* \le O(e^{\rho_{\tau(t)}})$, where $\rho_{\tau(t)} \geq \rho_t$. In this paper, we explore groups of \emph{time dilation} functions $\tau$ and their corresponding group action on the space of curves. We show that the family of Bregman Lagrangian functionals forms an orbit under the group action of time dilation; moreover, the family of Nesterov Lagrangian functionals (Section~\ref{Sec:NestFlow}) and the family of exp-Nesterov Lagrangian functionals (Section~\ref{Sec:ExpConv}) form isomorphic sub-orbits. We can therefore interpret the curves corresponding to family of accelerated methods $\faG$ as the result of speeding up (or traversing faster) the \emph{single} curve corresponding to accelerated gradient descent. The cost  for translating these faster curves into discrete-time algorithms (in addition to significant computational costs) is increasingly restrictive smoothness assumptions on the function.

We summarize in Figure~\ref{Fig:Summary} the relations between the objects we study in this paper. We see a consistent, almost parallel structure between continuous time (top layer) and discrete time (bottom layer). As the key conceptual message of this paper, we find there is a big difference in the nature of first-order equations (such as gradient flow) and second-order equations (such as accelerated gradient flow), due to the connection to the Lagrangian framework. Working with second-order equations provides better results in both continuous and discrete time.
\begin{sidewaysfigure}[!ph]
    \includegraphics[width=\textwidth]{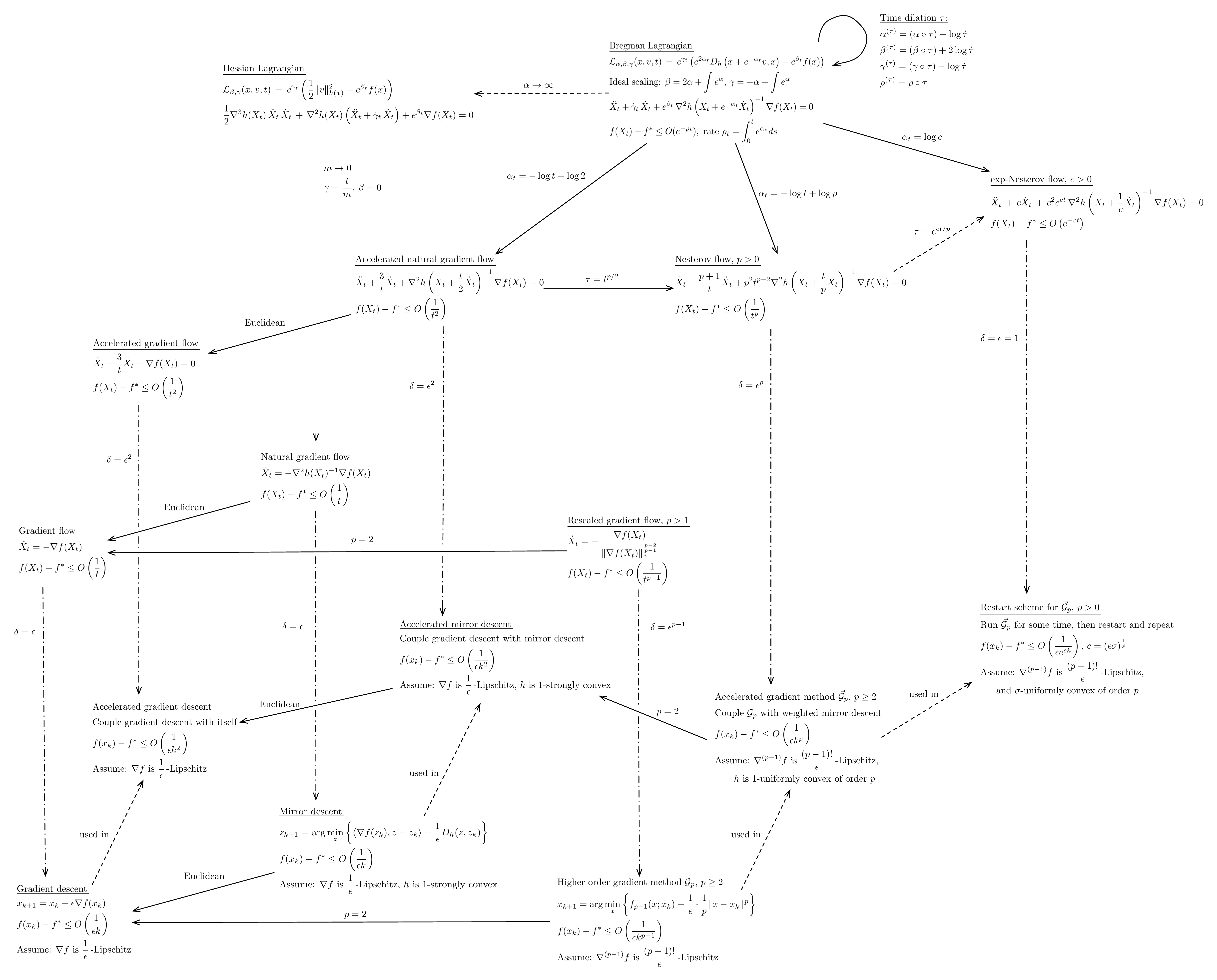}
    \caption{{\small Top layer is continuous time (first and second order equations), bottom layer is discrete time (first and second order algorithms).}}
    \label{Fig:Summary}
\end{sidewaysfigure}

\subsection{Notation and Preliminaries}
\label{Sec:Notation}
We formalize our setting and recall some basic definitions.
Our objective is to minimize a  {\em convex} objective function $f \colon \X \to \R$, which means the graph of $f$ lies above any tangent hyperplanes:
\begin{align}\label{Eq:Conv}
f(y) \ge f(x) + \langle \nabla f(x), y-x \rangle 
\end{align}
or equivalently, any intermediate value is at most the average value ({\em Jensen's inequality}):
\begin{align*}
f\big(\lambda x + (1-\lambda) y\big) \,\le\, \lambda f(x) + (1-\lambda) f(y) 
\end{align*}
for all $x,y \in \X$ and $0 \le \lambda \le 1$.
We assume $f$ is smooth and continuously differentiable as many times as necessary.

%
The domain $\X$ is an open convex subset of a vector space, say $\X \subseteq \R^d$ (we take $\X = \R^d$ for simplicity). In particular, we take the point of view that $\X$ is a manifold. 
%
%
%
In particular, any point $x \in \X$ has a {\em tangent space} $\TxX$ (vector space of instantaneous displacements $v$ such that $x+\varepsilon v \in \X$ for small $\varepsilon > 0$);
Since $\X \subseteq \R^d$, we can identify $\TxX$ with $\X$ (or $\R^d$) itself. But now we can have an interesting mixing of timescales between the points $x \in \X$ and the vectors $v \in \TxX$, which are now ``promoted'' from $\epsilon$ to the standard timescale. Indeed, we will see in Section~\ref{Sec:NestFlow} that Nesterov's acceleration technique involves the choice of mixing $x$ and $v$ using $\varepsilon = \frac{t}{p}$---which is counterintuitive because it is increasing, and yields sublinear rate of convergence. On the other hand, the exponential variant of Nesterov uses $\varepsilon = \frac{1}{c}$---which is more reasonable, and yields linear rate of convergence.

The {\em cotangent space} $\cTxX$ is the dual vector space to $\TxX$ (the space of linear functional $\phi$ on $\TxX$). 
For example, the {\em gradient} $\nabla f(x) \in \cTxX$ is a covector, which defines the directional derivative of $f$: 
\begin{align}\label{Eq:DirDer}
\langle \nabla f(x), v \rangle \,\equiv\, f'(x;v) \,:=\, \lim_{\epsilon \to 0} \frac{f(x+\epsilon v) - f(x)}{\epsilon}.
\end{align}
Because we use the $\ell_2$-norm, we can identify $\TxX \cong \cTxX$ by the identity map (which we implicitly use in~\eqref{Eq:GradDesc}); but note that conceptually, they are different spaces. Similarly, the dual norm $\|\cdot\|_*$ is also the same $\ell_2$-norm, but we will maintain the distinction of $\|\cdot\|_*$.

We say $f$ is {\em $L$-smooth of order $\ell \ge 0$} if $\nabla^\ell f$ is $L$-Lipschitz (and $\nabla^{(\ell+1)}$ is continuous):
\begin{align}\label{Eq:Smooth}
\|\nabla^\ell f(x) - \nabla^\ell f(y)\|_\ast \,\le\, L\|x-y\|.
\end{align}
The case $\ell = 0$ means $f$ is Lipschitz, and $\ell = 1$ is the usual smoothness definition ($\nabla f$ is Lipschitz).
We say $h$ is {\em $\sigma$-uniformly convex of order $p \ge 2$} ($p=2$ is strong convexity) if:
\begin{align}\label{Eq:UC}
D_h(y,x) \,\ge\, \frac{\sigma}{p} \|y-x\|^p.
\end{align}
where 
$D_h(y,x) = h(y) - h(x) - \langle \nabla h(x), y-x \rangle \ge 0$ is the Bregman divergence induced by a strictly convex {\em distance generating function} $h \colon \X \to \R$.

\subsection{Gradient algorithms}
\label{Sec:IntroGrad}

We review gradient-based algorithms that correspond to first-order equations in continuous time. 

\subsubsection{Mirror descent and natural gradient flow}
The {\em mirror descent} algorithm
\begin{align}\label{Eq:MD}
x_{k+1} &\,=\, \arg\min_v \,\left\{ \langle \nabla f(x_k), x \rangle + \frac{1}{\epsilon} D_h(x,x_k) \right\}
\end{align}
which can be written as
\begin{align*} 
   x_{k+1} &= x_k + v_k \notag\\
   v_k &=  \arg\min_{v} \left\{ f(x_k) + \langle \nabla f(x_k), v \rangle + \frac{1}{\epsilon} \,D_h(x_k + v, x_k) \right\} 
\end{align*}
measures displacement by the Bregman divergence.
Note that gradient descent~\eqref{Eq:GradDescOpt} takes $h(x) = \frac{1}{2}\|x\|^2$, $\X = \R^d$, and the classical multiplicative weight update uses $h(x) = -\sum x_i \log x_i$, $\X$ is the probability simplex.
In general, there are often suitable choices of $h$ that confer some computational gain in practice (e.g., milder dimension dependence). 
Similar to~\eqref{Eq:GradDescRate}, mirror descent has convergence rate
\begin{align}\label{Eq:MirrDescRate}
f(x_k) - f^* \leq O \left(\frac{1}{\epsilon k}\right)
\end{align}
when $\nabla f$ is $\frac{1}{\epsilon}$-Lipschitz, and $h$ is strongly convex (i.e., uniformly convex of order $2$).

In continuous time, mirror descent corresponds to (with $\delta = \epsilon$, $t = \epsilon k$) {\em natural gradient flow}:
\begin{align}\label{Eq:NatGradFlow1}
\dot X_t = -\,\nabla^2 h(X_t)^{-1} \, \nabla f(X_t)
\end{align}
which can be cast as the solution to the optimization problem
\begin{align}\label{Eq:NatGradFlowOpt1}
\dot X_t = \arg\min_{v} \left\{ f(X_t) + \langle \nabla f(X_t), v \rangle + \frac{1}{2} \|v\|_{h(X_t)}^2 \right\}
\end{align}
Natural gradient flow is a steepest descent direction when the metric in $\X$ is induced by the Hessian $\nabla^2 h$. Thus, mirror descent~\eqref{Eq:MD} can be seen as an alternative discretization of the natural gradient flow, another being Amari's~\cite{Amari1988} {\em natural gradient descent}\footnote{This equivalence has also been observed by~\cite{Raskutti15}; see Appendix~\ref{App:MD} for further discussion.}:
\begin{align}\label{Eq:NatGradDesc1}
x_{k+1} = x_k - \epsilon \nabla^2 h(x_k)^{-1} \nabla f(x_k).
\end{align}
Similar to \eqref{Eq:NatGradFlowOpt1}, we can interpret \eqref{Eq:NatGradDesc1} as the solution to the optimization problem
\begin{align*}
x_{k+1} ~&=~ x_k + v_k \notag\\
v_k & = \arg\min_{v} \left\{ f(x_k) + \langle \nabla f(x_k), v \rangle + \frac{1}{\epsilon} \cdot \frac{1}{2} \|v\|_{h(x_k)}^2 \right\}.  
\end{align*}
where $\|v\|_{h(x)}^2 = \langle \nabla^2 h(x) v, v \rangle$ is the Hessian norm induced by $h$. Furthermore, like gradient flow~\eqref{Eq:GradFlowRate}, natural gradient flow has convergence rate
\begin{align}\label{Eq:NatGradFlowRate}
f(X_t) - f^* \leq O \left(\frac{1}{t}\right)
\end{align}

\subsubsection{Newton's method and Newton's flow}
 {\em Newton's method} optimizes a quadratic approximation of the objective function $f$:
  \begin{align}\label{Eq:NewtReg}
 x_{k+1} = \arg\min_x \,\left\{ f(x_k) + \langle \nabla f(x_k), x \rangle + \frac{1}{2\epsilon} \langle \nabla^2 f(x_k)(x-x_k), x-x_k \rangle\right\}
  \end{align}
can be written explicitly as
\begin{align}\label{Eq:Newt}
x_{k+1} = x_k - \epsilon \nabla^2 f(x_k)^{-1} \, \nabla f(x_k).
\end{align}
The original Newton's method corresponds to $\epsilon = 1$, but there have been many proposed choices of step size $\epsilon$ to improve stability and ensure convergence (e.g., see~\cite{NesterovPolyak06} and references therein).

We can also interpret~\eqref{Eq:Newt} as natural gradient descent~\eqref{Eq:NatGradDesc1} when $h=f$. Thus, in continuous time it corresponds to {\em Newton's flow}:\footnote{Note, Newton's flow is explicitly solvable: $X_t = \nabla f^*(e^{-t} \nabla f(X_0))$, where $f^*$ is the convex dual of $f$}
\begin{align}\label{Eq:NewtFlow}
\dot X_t = -\nabla^2 f(X_t)^{-1} \nabla f(X_t)
\end{align}
which is natural gradient flow \eqref{Eq:NatGradFlow1} with $f = h$.
However, convergence results for the scheme~\eqref{Eq:Newt} are difficult to obtain and have all been local. Only in special cases (e.g., self-concordance, a local Lipschitz condition on $\nabla^2 f$) are we able to have any strong guarantee on Newton's method.

\subsubsection{Cubic-regularized Newton's method and rescaled gradient flow}
To address this issue, Nesterov and Polyak~\cite{NesterovPolyak06} proposed {\em cubic-regularized Newton's method}, which optimizes a second-order approximation of $f$ plus regularization:
{\fontsize{10.5}{0}
\begin{align}\label{Eq:CubNewt}
x_{k+1} \,=\, \arg\min_x \,\left\{ f(x_k) + \langle \nabla f(x_k), x-x_k \rangle + \frac{1}{2} \langle \nabla^2 f(x_k)(x-x_k), x-x_k \rangle + \frac{1}{3\epsilon} \|x-x_k\|^3 \right\}
\end{align}}
They showed~\cite[Theorem~4]{NesterovPolyak06} that if $\nabla^2 f$ is $\frac{2}{\epsilon}$-Lipschitz, then~\eqref{Eq:CubNewt} has convergence rate guarantee
\begin{align}\label{Eq:CubNewtRate}
f(x_k) - f^* \le \frac{27 \|x_0-x^*\|^2}{\epsilon (k+3)^2} = O\left(\frac{1}{\epsilon k^2} \right).
\end{align}
As mentioned in Section~\ref{Sec:IntroTime}, we show (Section~\ref{Sec:Rescaledgrad}) that in continuous time (with $\delta = \epsilon^{\frac{1}{2}}$, $t^2 = \epsilon k^2$),~\eqref{Eq:CubNewt} corresponds to the rescaled gradient flow:
\begin{align}\label{Eq:CNMFlow}
 \quad \dot X_t = \frac{\nabla f(X_t)}{\|\nabla f(X_t)\|_*^{1/2}}
\end{align}
and that~\eqref{Eq:CNMFlow} has matching convergence rate
\begin{align}\label{Eq:CNMRate}
f(X_t) - f^* \leq O \left(\frac{1}{t^2}\right).
\end{align}

Finally, we note that by adding a regularization term in Newton's method~\eqref{Eq:CubNewt}, Nesterov and Polyak changed the problem from discretizing a Newton's flow~\eqref{Eq:NewtFlow} to rescaled gradient flow~\eqref{Eq:CNMFlow}. Thus, the two variants~\eqref{Eq:Newt},~\eqref{Eq:CubNewt} differ quite a bit in continuous time.

\subsection{Accelerated gradient algorithms}
\label{Sec:NesterovAcc}

We review accelerated gradient algorithms that correspond to second-order equations in continuous time.

\subsubsection{Accelerated gradient descent}
\label{Sec:AGD}

Nesterov's accelerated gradient descent~\cite{Nesterov83,Nesterov04} improves the performance of gradient descent~\eqref{Eq:GradDesc} from $O(1/\epsilon k)$ to the optimal $O(1/\epsilon k^2)$. This gain is achieved not by strengthening the assumption on $f$, but by
incorporating the displacement $x_k - x_{k-1}$ to shift the point at which we query the gradient $\nabla f$ (thus, this method is often referred to as gradient descent ``with momentum'').
The algorithm ~\cite[(2.2.6)]{Nesterov04} maintains three sequences $\{x_k\}, \{y_k\}, \{z_k\}$ and proceeds as follows. For any $y_0 = z_0 \in \X$, $k \ge 0$:
\begin{subequations}\label{Eq:AGD}
\vspace{-8pt}
\begin{align}
x_k &= \: \tau_k \, z_k \,+\, (1-\tau_k) \, y_k,    \label{Eq:AGDx} \\
y_{k+1} &= \: x_k  - \epsilon \nabla f(x_k),    \label{Eq:AGDy}  \\
z_{k+1} &= \: z_k - \frac{\epsilon}{\tau_k} \nabla f(x_k).    \label{Eq:AGDz}
\end{align}
\end{subequations}
Here $\epsilon > 0$ is the step size, and $\tau_k \in (0,1)$ is defined recursively by $\tau_{-1} = 1$ and the rule, for $k \ge 0$:
\begin{align}\label{Eq:alpha}
\frac{\tau_k^2}{1-\tau_k} = \tau_{k-1}^2.
\end{align}
We can also see that $\tau_k = 2/k + o(1/k)$, for indeed we have $\frac{\tau_k^2}{1-\tau_k} = \frac{4/k^2}{1-2/k} = \frac{4}{k(k-2)} \approx \frac{4}{(k-1)^2} = \tau_{k-1}^2$.
Nesterov showed~\cite[Theorem~2.2.2]{Nesterov04} that when $\nabla f$ is $\frac{1}{\epsilon}$-Lipschitz, then~\eqref{Eq:AGD} satisfies: 
\begin{align}\label{Eq:AGDRate}
f(y_k) - f^* \le \frac{4\|x_0-x^*\|^2}{\epsilon \, (k+2)^2} = O\left(\frac{1}{\epsilon k^2} \right),
\end{align}
which improves the $O(1/\epsilon k)$ rate~\eqref{Eq:GradDescRate} of gradient descent, and matches the lower bound.

As pointed out by Su, Boyd and Candes~\cite{SuBoydCandes14}, in continuous time accelerated gradient descent corresponds to the second order equation:
\begin{align}\label{Eq:SBC0}
\ddot X_t + \frac{3}{t} \dot X_t + \nabla f(X_t) = 0
\end{align}
with time scaling $\delta = \epsilon^{\frac{1}{2}}$, $t^2 = \epsilon k^2$, and matching convergence rate~\cite[Theorem~3.2]{SuBoydCandes14}:
\begin{align}\label{Eq:SBC0Rate}
f(X_t) - f^* \leq O \left(\frac{1}{t^2}\right).
\end{align}


\subsubsection{Accelerated mirror descent}
\label{Sec:AMD}

In~\cite{Nesterov05}, Nesterov proposed {\em accelerated mirror descent}, which proceeds much like the Euclidean case~\eqref{Eq:AGD}, except we replace the $z$-update~\eqref{Eq:AGDz} by a (weighted) mirror descent step~\eqref{Eq:MD}.
The algorithm~\cite[(3.11)]{Nesterov05} maintains three sequences $\{x_k\},\{y_k\},\{z_k\}$ and proceeds as follows. For any $x_0 \in \X$, $k \ge 0$:
\begin{subequations}\label{Eq:AMD}
\begin{align}
y_k &= \: \arg\min_y \, \;\left\{ f(x_k) + \langle \nabla f(x_k), y-x_k \rangle + \frac{1}{2\epsilon} \|y-x_k\|^2 \right\}  \label{Eq:AMDy}   \\
z_k &= \: \arg\min_z \, \left\{ \sum_{i=0}^k \frac{i+1}{2} \big[ f(x_i) + \langle \nabla f(x_i), z-x_i \rangle \big] \,+\, \frac{1}{\epsilon}\cdot \frac{1}{\sigma} D_h(z,x_0) \right\}   \label{Eq:AMDz}   \\
x_{k+1} &= \: \frac{2}{k+3} \, z_k \,+\, \frac{k+1}{k+3} \, y_k    \label{Eq:AMDx}
\end{align}
\end{subequations}
Under the same assumption as mirror descent ($\nabla f$ is $\frac{1}{\epsilon}$-smooth and $h$ is $\sigma$-strongly convex), the accelerated version~\eqref{Eq:AMD} has improved (optimal) convergence~\cite[Theorem~2]{Nesterov05}:
\begin{align}\label{Eq:AMDRate}
f(y_k) - f^* \le \frac{4D_h(x^*,x_0)}{\epsilon \sigma (k+1)(k+2)} = O\left(\frac{1}{\epsilon k^2}\right).
\end{align}
Note, using the equivalence of mirror descent as the cascaded version of dual averaging, we can write the update~\eqref{Eq:AMDz} above recursively using the standard mirror descent algorithm:\footnote{This observation has also been reported by~\cite{AO14}.}
\begin{align}\label{Eq:AMDz2}
z_k &= \: \arg\min_z \, \left\{ \frac{k+1}{2} \langle \nabla f(x_k), z \rangle \,+\, \frac{1}{\epsilon} \cdot \frac{1}{\sigma}D_h(z,z_{k-1}) \right\}.
\end{align}
In the Euclidean case (when $h = \frac{1}{2}\|\cdot\|_2^2$ and setting $\sigma = 1$), the update~\eqref{Eq:AMDz2} above simplifies to the explicit rule $z_k = z_{k-1} - \frac{\epsilon(k+1)}{2} \nabla f(x_k)$, recovering accelerated gradient descent~\eqref{Eq:AGD}.

Similar to~\eqref{Eq:SBC0}, accelerated mirror descent corresponds to the second order equation:
\begin{align}\label{Eq:SBC0}
\ddot X_t + \frac{3}{t} \dot X_t + \nabla^2 h\left(X_t+\frac{t}{2}\dot X_t\right)^{-1} \nabla f(X_t) = 0
\end{align}
with time scaling $\delta = \epsilon^{\frac{1}{2}}$, $t^2 = \epsilon k^2$, and matching convergence rate:
\begin{align}\label{Eq:SBC0Rate}
f(X_t) - f^* \leq O \left(\frac{1}{t^2}\right).
\end{align}

\subsubsection{Accelerated cubic-regularized Newton's method}

In~\cite{Nesterov08}, Nesterov proposed the {\em accelerated cubic-regularized Newton's method}, which proceeds as~\eqref{Eq:AGD}, except we replace the $y$-update~\eqref{Eq:AGDy} by cubic-regularized Newton's method~\eqref{Eq:CubNewt}.
The algorithm~\cite[(4.8)]{Nesterov08} maintains three sequences $\{x_k\},\{y_k\},\{z_k\}$ and proceeds as follows. For any $x_0 \in \X$, $k \ge 0$:
\begin{subequations}\label{Eq:ACN}
{\fontsize{10.5}{0}
\begin{align} 
y_k &= \arg\min_y\left\{ f(x_k) + \langle \nabla f(x_k), y-x_k \rangle + \frac{1}{2} \langle \nabla^2 f(x_k)(y-x_k), y-x_k \rangle + \frac{1}{3\epsilon} \|y-x_k\|^3 \right\}   \label{Eq:ACNy}  \\  
z_k &= \arg\min_z\left\{\sum_{i=1}^k \frac{i(i+1)}{2} \big[ f(y_i) + \langle \nabla f(y_i), z-y_i \rangle \big] + \frac{2}{\epsilon} \|z-x_0\|^3 \right\},  \label{Eq:ACNz}   \\
x_{k+1} &=\: \frac{3}{k+3} z_k + \frac{k}{k+3} y_k.    \label{Eq:ACNx}
\end{align}}
\end{subequations}
Under the same assumption as the cubic-regularized Newton's method, $\nabla^2 f$ is $\frac{2}{\epsilon}$-Lipschitz, the accelerated algorithm~\eqref{Eq:ACN} has convergence rate~\cite[Theorem~2]{Nesterov08}:
\begin{align}\label{Eq:ACNRate}
f(y_k) - f^* \le \frac{14\|x_0-x^*\|^3}{\epsilon k(k+1)(k+2)} = O\left(\frac{1}{\epsilon k^3}\right).
\end{align}

As noted in Section~\ref{Sec:IntroTime}, we show (Section~\ref{Sec:DiscToCts}) the accelerated algorithm~\eqref{Eq:ACN} corresponds to the following differential equation:
\begin{align}\label{Eq:SBC1}
\ddot X_t + \frac{4}{t} \dot X_t + 9t \nabla^2 h\left(X_t+\frac{t}{3}\dot X_t\right)^{-1} \nabla f(X_t) = 0
\end{align}
with $\delta = \epsilon^{\frac{1}{3}}$, $t^3 = \epsilon k^3$, where $h(x) = \frac{1}{3}\|x\|^3$ in~\eqref{Eq:ACN}. Furthermore, \eqref{Eq:SBC1} has matching convergence rate:
\begin{align}\label{Eq:SBC1Rate}
f(X_t) - f^* \leq O \left(\frac{1}{t^3}\right).
\end{align}

\section{Higher-order gradient methods and rescaled gradient flows}
\label{Sec:HigherGrad}

We study the family of higher-order gradient methods $\G_p \in \fG$ in discrete time, which corresponds to first-order rescaled gradient flow in continuous time, with time step $\delta = \epsilon^{\frac{1}{p-1}}$, and matching convergence rate $O(1/t^{p-1}) = O(1/\epsilon k^{p-1})$. 

\paragraph{Surrogate optimization.}
Recall in discrete time, many optimization algorithms proceed by minimizing a {\em surrogate function}:\footnote{See~\cite{Lange14} and the references therein for more information on majorization-minimization principle in optimization.}
\begin{align}\label{Eq:MM}
x_{k+1} = \arg\min_{x \in \X} \: g(x; x_k).
\end{align}
Here $g(x;x_k)$ is a surrogate function that {\em majorizes} the objective function $f(x)$, which means for any reference point $x_k \in \X$ and for any point $x \in \X$, it satisfies the inequality:
\begin{align}\label{Eq:Surr}
g(x;x_k) \ge f(x)
\end{align}
and equality holds at $x = x_k$.
The property~\eqref{Eq:Surr} above implies the algorithm~\eqref{Eq:MM} is a {\em descent method}, which means the function value $f(x_k)$ decreases along the iteration of the algorithm:
\begin{align}\label{Eq:Descent}
f(x_{k+1})  \stackrel{\eqref{Eq:Surr}}{\le} g(x_{k+1}; x_k) \stackrel{\eqref{Eq:MM}}{\le} g(x_k; x_k) \,=\, f(x_k).
\end{align}
We can minimize $f$ by finding an appropriate surrogate function $g(x;x_k)$ that is more tractable to minimize, and performing the descent algorithm~\eqref{Eq:MM}.
Moreover, under various assumptions, we can quantify the decrease in the function value~\eqref{Eq:Descent}, resulting in a rate of convergence for the algorithm~\eqref{Eq:MM}.


\paragraph{Surrogate via Taylor expansion.}
A natural technique for constructing a surrogate function $g(x;y)$ is to use the Taylor approximation of $f$ at $x$ from the reference point $y$ 
{\fontsize{10.47pt}{0}
\begin{align}\label{Eq:Taylor}
f_{p-1}(x; y) \,=\, \sum_{i=0}^{p-1} \frac{1}{i!} \nabla^i f(y) (x-y)^i
\,=\, f(y) + \langle \nabla f(y), x-y \rangle + \cdots + \frac{1}{(p-1)!} \nabla^{p-1} f(y) (x-y)^{p-1}.
\end{align}}
If $f$ is $L$-Lipschitz of order $p-1$~\eqref{Eq:Smooth}, then we have a bound on the approximation error of $f_{p-1}$:
\begin{align}\label{Eq:TayErr}
|f(x) - f_{p-1}(x;y)| \,\le\, \frac{L}{p!} \|x-y\|^{p}.
\end{align}
Therefore, we have the family of (regularized) {\em Taylor surrogate functions}, for $p \ge 1$:
\begin{align}\label{Eq:SurrTay}
g_p(x;y) \,:=\, f_{p-1}(x;y) + \frac{L}{p!} \|x-y\|^{p}.
\end{align}
Note also the tangency property $g_{p}(x;x) = f_{p-1}(x;x) = f(x)$, so $g_p$ is indeed a surrogate function.

\subsection{Higher-order gradient method}
\label{Sec:HigherGrad1}

The Taylor surrogate functions~\eqref{Eq:SurrTay} give rise to {\em higher-order gradient methods} $\G_p \in \fG$, defined by the update equation:
\begin{align}\label{Eq:HigherGrad}
x_{k+1} = \arg\min_{x } \,\left\{ f_{p-1}(x;x_k) + \frac{1}{\epsilon} \cdot \frac{1}{p} \|x-x_k\|^p \right\}.
\end{align}
The $p=2$ case gives the gradient descent algorithm~\eqref{Eq:GradDesc}, and $p=3$ is Nesterov and Polyak's~\cite{NesterovPolyak06} cubic-regularized Newton's method~\eqref{Eq:CubNewt}. Note also that $p=1$ gives a constant sequence, so we only consider $p \ge 2$.
In~\eqref{Eq:HigherGrad}, we write the Lipschitz constant $L = \frac{(p-1)!}{\epsilon} = O(\frac{1}{\epsilon})$ in terms of the {\em step size} $\epsilon > 0$, representing the discretization parameter of the algorithm.\footnote{This assumption vanishes in continuous time, since $\frac{1}{\epsilon} \to \infty$ as $\epsilon \to 0$.} 
Our discussion above says that if $f$ is $\frac{(p-1)!}{\epsilon}$-smooth of order $p-1$, then the $p$-th gradient algorithm $\G_p$~\eqref{Eq:HigherGrad} is a descent method, since $f(x_{k+1}) \le f(x_k)$.

Moreover, we can use the convexity structure of $f$ to further ensure a quantitative decrease in the {\em residual} $\delta_k = f(x_k) - f^* \ge 0$:
\begin{lemma}\label{Lem:Resd}
If $f$ is convex and $\frac{(p-1)!}{\epsilon}$-smooth of order $p-1$, then the following holds for \eqref{Eq:HigherGrad}:
\begin{align}\label{Eq:Resd}
\delta_{k+1} \,\le\, \delta_k \,-\, \frac{\epsilon^{\frac{1}{p-1}}}{R^{\frac{p}{p-1}}} \cdot \delta_k^{\frac{p}{p-1}}
\end{align}
where $R = \sup_{f(x) \le f(x_0)} \|x-x^*\|$ is the {\em radius} of the level set from the initial point $x_0$.
\end{lemma}
The proof of Lemma~\ref{Lem:Resd} is in Appendix~\ref{App:ConvHigherGrad}. Using the {\em (discrete time) energy functional} $\E_k = \delta_k^{-\frac{1}{p-1}}$,  we obtain following convergence rate for $\G_p$, generalizing~\cite[Theorem~1]{Nesterov08}:
\begin{theorem}\label{Thm:HigherGradRate}
If $f$ is convex and $\frac{(p-1)!}{\epsilon}$-smooth of order $p-1$, then the following holds for \eqref{Eq:HigherGrad}:
\begin{align}\label{Eq:HigherGradRate}
f(x_k) - f^* \le \frac{(p-1)^{p-1} R^p}{\epsilon k^{p-1}} = O\left(\frac{1}{\epsilon k^{p-1}}\right).
\end{align}
\end{theorem}

Thus, the family of higher-order gradient methods $\G_p \in \fG$ in discrete time has a nice sequential structure. In particular, there is a consistent pattern whereby as $p$ progresses in $\{2,3,\dots\}$, the polynomial convergence rate $O(1/\epsilon k^{p-1})$ decreases, but at the cost of increasingly strict $(p-1)$-st order smoothness assumption on $f$. This suggests there is a fundamental tradeoff in discrete time between speed of convergence and strength of hypothesis required.


\subsection{Rescaled gradient flow}
\label{Sec:Rescaledgrad}
In continuous time, gradient flow~\eqref{Eq:GradFlow} is a member ($p=2$) of a family of {\em rescaled gradient flows}:
\begin{align}\label{Eq:RescGradFlow}
\dot X_t = -\,\frac{\nabla f(X_t)}{\|\nabla f(X_t)\|_*^{\frac{p-2}{p-1}}}.
\end{align}
When $\nabla f(X_t) = 0$, we define the right hand side of~\eqref{Eq:RescGradFlow} to be $0$. Note that~\eqref{Eq:RescGradFlow} implies that the magnitude of the velocity at some time is proportional to (some power of) the gradient at that point:
\begin{align}\label{Eq:RescGradFlow3}
\|\dot X_t\| = \|\nabla f(X_t)\|_*^{\frac{1}{p-1}}
\end{align}
so we can also equivalently write the rescaled gradient flow as
\begin{align}\label{Eq:RescGradFlow2}
\|\dot X_t\|^{p-2} \, \dot X_t = -\nabla f(X_t).
\end{align}
Furthermore, \eqref{Eq:RescGradFlow2} is the optimality condition for the following optimization problem:
\begin{align}\label{Eq:RescGradFlowOpt}
\dot X_t = \arg\min_v \,\left\{ f(X_t)+ \langle \nabla f(X_t), v \rangle + \frac{1}{p} \|v\|^p \right\}.
\end{align}
Thus, we can view the rescaled gradient flow~\eqref{Eq:RescGradFlow} as a generalization of gradient flow that replaces the squared norm in the optimization interpretation~\eqref{Eq:GradFlowOpt} by the $p$-th power of the norm, for $p \ge 2$.

If $X_t$ is a curve that evolves following the rescaled gradient flow~\eqref{Eq:RescGradFlow}, then, using the convexity of $f$, the {\em energy functional}:
\begin{align}\label{Eq:RescGradFlowEnergy}
\E_t = (f(X_t) - f^*)^{-\frac{1}{p-1}}
\end{align}
increases linearly over time ($\dot \E_t \ge \frac{1}{p-1} R^{-\frac{p}{p-1}}$, see Appendix~\ref{App:RescaledGrad} for details), which implies the following rate of convergence:
\begin{theorem}\label{Thm:RescaledGradFlow} If $f$ is convex with bounded level sets, then the following holds for \eqref{Eq:RescGradFlow}:
\begin{align}\label{Eq:RescGradFlowRate}
f(X_t) - f^* \le \frac{(p-1)^{p-1} R^p}{t^{p-1}} = O\left(\frac{1}{t^{p-1}}\right)
\end{align}
\end{theorem}

The result~\eqref{Eq:RescGradFlowRate} above matches the discrete time convergence rate~\eqref{Eq:HigherGradRate} of the $p$-th order gradient method $\G_p$. Moreover, notice we use the same energy functional~\eqref{Eq:RescGradFlowEnergy} as in discrete time.

\subsection{The relation between $\epsilon$ and $\delta$}
\label{Sec:EpsDelta}

The matching convergence rates~\eqref{Eq:HigherGradRate},~\eqref{Eq:RescGradFlowRate} suggest the following {\em time scaling} relation between continuous time $t \ge 0$ and discrete time $k \in \{2,3,\dots\}$:
\begin{align}\label{Eq:TimeScale}
t = \epsilon^{\frac{1}{p-1}} k.
\end{align}
So in one step of discrete time $k \mapsto k+1$, the continuous time $t \equiv t_k$ increments by:
\begin{align}\label{Eq:TimeScale2}
\delta \,\equiv\, \frac{dt}{dk} \,:=\, \frac{t_{k+1} - t_k}{(k+1)-k}
          \stackrel{\eqref{Eq:TimeScale}}{=}
          \epsilon^{\frac{1}{p-1}}.
\end{align}
The scaling above suggests the interpretation of the $p$-th gradient algorithm $\G_p$ as a discretization of the rescaled gradient flow with time step $\delta = \epsilon^{\frac{1}{p-1}}$.

We can understand this scaling phenomenon more explicitly by starting from the continuous time view. Suppose we have a continuous time curve $X_t$ (such as the rescaled gradient flow~\eqref{Eq:RescGradFlow}), and we want to {\em discretize} it with time step $\delta > 0$. This means we
build a discrete time sequence $x_k$ obtained by taking a snapshot of $X_t$ every $\delta$ increment of time, namely:
\begin{align}\label{Eq:CtsDisc}
X_t = x_k~~~\Rightarrow~~~X_{t+\delta} = x_{k+1}.
\end{align}
Notice, to go from $x_k$ to $x_{k+1}$ in~\eqref{Eq:CtsDisc} above, we have to first let $X_t$ evolve in continuous time to $X_{t+\delta}$, which we then use as the value for $x_{k+1}$.

However, to discretize is to build a discrete time algorithm, which means we can only define $x_{k+1}$ in terms of the previous discrete iterates $x_k, x_{k-1}, \dots$, without invoking the continuous time curve $X_t$. 
Thus, we have to approximate the continuous time evolution from $X_t$ to $X_{t+\delta}$.

For the first-order rescaled gradient flow~\eqref{Eq:RescGradFlow}, we approximate $X_{t+\delta}$ using a linear approximation:
\begin{align}\label{Eq:CtsDisc1}
X_{t+\delta} \approx X_t + \delta \dot X_t
\end{align}
with an error of order $o(\delta)$. Equivalently, we replace $\dot X_t$ in~\eqref{Eq:RescGradFlowOpt} by the discrete time difference:\footnote{We can also start by replacing $\dot X_t$ in~\eqref{Eq:RescGradFlow} by~\eqref{Eq:CtsDisc1} and achieve the same conclusion. Namely,~\eqref{Eq:RescGradFlow} becomes $\frac{1}{\delta}(x_{k+1}-x_k) = -\nabla f(x_k)/\|\nabla f(x_k)\|_*^{\frac{p-2}{p-1}}$, or equivalently, $\|x_{k+1}-x_k\|^{p-2} (x_{k+1}-x_k) = -\delta^{p-1}\nabla f(x_k)$, which is~\eqref{Eq:RescGradFlowOptDisc2}.}
\begin{align}\label{Eq:RescGradFlowOptDisc}
\frac{x_{k+1}-x_k}{\delta} = \arg\min_v \,\left\{ \langle \nabla f(x_k), v \rangle + \frac{1}{p} \|v\|^p \right\}.
\end{align}
Or, writing $v = \frac{x-x_k}{\delta}$, the above can be written as:
\begin{align}
x_{k+1} 
&= \arg\min_x \,\left\{ f(x_k) + \langle \nabla f(x_k), x-x_k\rangle + \frac{1}{\delta^{p-1}} \cdot \frac{1}{p} \|x-x_k\|^p \right\}. \label{Eq:RescGradFlowOptDisc2}
\end{align}
Thus, we obtain an algorithm similar to the $p$-th gradient method~\eqref{Eq:RescGradFlowOpt}, where the step size $\epsilon$ is given by $\delta^{p-1}$, consistent with the time scaling~\eqref{Eq:TimeScale2}.
We can interpret the $p$-th gradient method~\eqref{Eq:HigherGrad} as a particular discretization technique of the rescaled gradient flow~\eqref{Eq:RescGradFlow}, which replaces the first-order approximation of $f$ in the ``naive'' discretization~\eqref{Eq:RescGradFlowOptDisc2} by the $(p-1)$-st order approximation~\eqref{Eq:HigherGrad}. By doing so, as well as assuming $(p-1)$-st order smoothness of $f$, the resulting discrete time algorithm $\G_p$ has a $O(1/\epsilon k^{p-1})$ convergence rate which matches the $O(1/t^{p-1})$ bound in continuous time.


\section{Accelerated higher-order gradient methods}
\label{Sec:AccHigherGrad}

In Section~\ref{Sec:NesterovAcc}, we see the simple pattern in Nesterov's constructions of accelerated methods~\cite{Nesterov04,Nesterov05,Nesterov08}:
\begin{center}
{\em To accelerate an algorithm, couple it with a (suitably weighted) mirror descent step.}
\end{center}

In this section, we extend Nesterov's technique to accelerate all higher-order gradient methods. The accelerated $p$-th order gradient method $\aG_p$ is obtained by coupling $\G_p$ with a mirror descent step weighted by a polynomial of order $p-1$. The accelerated algorithm $\aG_p$ has an improved convergence rate $O(1/\epsilon k^p)$ under the same $(p-1)$-st order smoothness assumption as $\G_p$. Thus, just like $\fG$, the family of accelerated gradient methods $\aG_p \in \faG$ still maintains the nice sequential property of the polynomially decreasing convergence rates.\footnote{While preparing this paper, we became aware of an unpublished manuscript by Baes~\cite{Baes09} who extended Nesterov's technique of estimate sequence and constructed higher-order variants of the accelerated methods, essentially identical to ours. We nevertheless present our generalization of Nesterov's proof in order to highlight the basic structure.}


Throughout this section, we fix an integer $p \ge 2$, and assume $f$ is $\frac{(p-1)!}{\epsilon}$-smooth of order $p-1$~\eqref{Eq:Smooth}. For the mirror descent step, we assume the distance generating function $h$ is $\sigma$-strongly convex of order $p$~\eqref{Eq:UC}, where $\sigma > 0$ is a constant (we can normalize $\sigma = 1$).
For example, we can take $h$ to be the $p$-th power of the norm:
\begin{align}\label{Eq:Dist}
d_p(x) = \frac{1}{p} \|x-x_0\|^p
\end{align}
(for arbitrary reference point $x_0 \in \X$), which is $(\frac{1}{2})^{p-2}$-uniformly convex of order $p$~\cite[Lemma~4]{Nesterov08}.


\subsection{Accelerated $p$-th order gradient method}
\label{Sec:HigherGradAlg}
The {\em accelerated $p$-th order gradient method} $\aG_p$ maintains three sequences $\{x_k\},\{y_k\},\{z_k\}$ as follows. Starting from any $x_0 \in \X$, $k \ge 0$ the algorithm proceeds:
\begin{subequations}\label{Eq:AGM}
\begin{align}
y_k &= \arg\min_y \left\{ f_{p-1}(y;x_k) + \frac{2}{\epsilon} \cdot \frac{1}{p} \|y-x_k\|^p \right\}   \label{Eq:AGMy}  \\
z_k &= \arg\min_z \left\{ \sum_{i=0}^k Cp \, i^{(p-1)} \big[f(y_i) + \langle \nabla f(y_i), z-y_i \rangle \big] + \frac{1}{\epsilon} \cdot \frac{1}{\sigma} D_h(z,x_0) \right\}   \label{Eq:AGMz}  \\
x_{k+1} &= \frac{p}{k+p} z_k + \frac{k}{k+p} y_k  \label{Eq:AGMx}
\end{align}
\end{subequations}
where $i^{(p-1)} := i(i+1) \cdots (i+p-2)$ denotes the {\em rising factorial}, and $C  \le (4p)^{-p}$ is a constant.

Note, the $y$-update~\eqref{Eq:AGMy} above is the $p$-th gradient method, but with slightly larger regularization ($\frac{c}{\epsilon}$ with any $c > 1$; above is $c=2$). As noted in Section~\ref{Sec:NesterovAcc}, the $z$-update~\eqref{Eq:AGMz} above is given in a ``dual averaging'' form, because (for example, when $\X = \R^d$) we can write it explicitly:
\begin{align}\label{Eq:AGMzUp}
\nabla h(z_k) \,=\, \nabla h(x_0) - \epsilon \sigma C p \sum_{i=0}^k i^{(p-1)} \nabla f(y_i)
\,=\, \nabla h(z_{k-1}) - \epsilon \sigma C p k^{(p-1)} \nabla f(y_k).
\end{align}
Therefore, the $z$-update~\eqref{Eq:AGMz} can also be equivalently written recursively as a mirror descent step:
\begin{align}\label{Eq:AGMz2}
z_k = \arg\min_z \left\{ Cpk^{(p-1)} \langle \nabla f(y_k),z \rangle  + \frac{1}{\epsilon \sigma} D_h(z,z_{k-1}) \right\}.
\end{align}
But it turns out that the expanded form of the update~\eqref{Eq:AGMz} is more convenient for us.

\subsection{Convergence analysis}

We can justify the performance of the accelerated algorithm~\eqref{Eq:AGM} by following a straightforward generalization of Nesterov's arguments~\cite{Nesterov05,Nesterov08}, which proceeds as follows.
We first recall the following property for the $p$-th order gradient step in the $y$-update~\eqref{Eq:AGMy}. This lemma generalizes~\cite[Lemma~6]{Nesterov08}, and its proof is provided in Appendix~\ref{App:LemHigherGrad}.
\begin{lemma}\label{Lem:HigherGrad}
If $f$ is $\frac{(p-1)!}{\epsilon}$-smooth of order $p-1$, then the $y$-update~\eqref{Eq:AGMy} has the guarantee:
\begin{align}\label{Eq:HigherGradLem}
\langle \nabla f(y_k), \, x_k-y_k \rangle \,\ge\, \frac{1}{4} \, \epsilon^{\frac{1}{p-1}} \|\nabla f(y_k)\|_*^{\frac{p}{p-1}}.
\end{align}
\end{lemma}

\smallskip
With Lemma~\ref{Lem:HigherGrad} in hand, we can proceed as follows. Let $\psi_k$ (Nesterov's ``estimate function'') denote the objective function in the $z$-update~\eqref{Eq:AGMz}:
\begin{align}\label{Eq:EstFunc}
\psi_k(x) = Cp \sum_{i=0}^k i^{(p-1)} \big[f(y_i) + \langle \nabla f(y_i), x-y_i \rangle \big] + \frac{1}{\epsilon \sigma} D_h(x,x_0).
\end{align}
Since $f$ is convex, each term in the summation above is at most $Cp\, i^{(p-1)} f(x)$. Noting that $\sum_{i=0}^k  i^{(p-1)} = k^{(p)}/p$, this yields the following upper bound on the estimate function, for all $x \in \X$:
\begin{align}\label{Eq:EstFunc2}
\psi_k(x)
\,\le\,
Ck^{(p)} f(x) + \frac{1}{\epsilon \sigma} D_h(x,x_0).
\end{align}
Furthermore, the updates in~\eqref{Eq:AGM} are constructed in such a way that we also have the following guarantee.
\begin{proposition}\label{Prop:AGM1}
If $C \le (4p)^{-p}$, then for all $k \ge 0$:
\begin{align}\label{Eq:AGMProp1}
Ck^{(p)} \, f(y_k) \,\le\, \psi_k^* \,:=\, \min_x \psi_k(x).
\end{align}
\end{proposition}
\begin{proof}
We proceed via induction on $k \ge 0$. The base case $k = 0$ is trivial since both sides equal $0$. Now assume~\eqref{Eq:AGMProp1} holds for some $k \ge 0$; we will show it also holds for $k+1$.

Since $h$ is $\sigma$-uniformly convex of order $p$, the rescaled Bregman divergence $\frac{1}{\epsilon \sigma} D_h(x,x_0)$ is $\frac{1}{\epsilon}$-uniformly convex. Thus, the estimate function $\psi_k$~\eqref{Eq:EstFunc} is also $\frac{1}{\epsilon}$-uniformly convex of order $p$.
Since $z_k$ is the minimizer of $\psi_k$, this implies $\psi_k(x) \,\ge\, \psi_k^* + \frac{1}{\epsilon p} \|x-z_k\|^p$.
We then apply the inductive hypothesis~\eqref{Eq:AGMProp1} and use the convexity of $f$~\eqref{Eq:Conv} to obtain the bound, for all $x \in \X$:
\begin{align}\label{Eq:AGMProp2}
\psi_k(x) \,\ge\, Ck^{(p)} \big[ f(y_{k+1}) + \langle \nabla f(y_{k+1}), y_k-y_{k+1} \rangle \big] + \frac{1}{\epsilon p} \|x-z_k\|^p.
\end{align}
Then by adding the $(k+1)$-st term in the definition of $\psi_{k+1}$~\eqref{Eq:EstFunc} to both sides of~\eqref{Eq:AGMProp2}, we obtain:
\begin{align}\label{Eq:AGMProp3}
\psi_{k+1}(x)
\,\ge\,
C(k+1)^{(p)} \big[ f(y_{k+1}) + \big\langle \nabla f(y_{k+1}),\, x_{k+1}-y_{k+1} + \tau_k (x-z_k) \big\rangle \big] + \frac{1}{\epsilon p} \|x-z_k\|^p
\end{align}
where $\tau_k = \frac{p (k+1)^{(p-1)}}{(k+1)^{(p)}} = \frac{p}{k+p}$, and in the above we have also used the definition of $x_{k+1}$ as a convex combination of $y_k$ and $z_k$ with weight $\tau_k$~\eqref{Eq:AGMx}.

Note, the first term in~\eqref{Eq:AGMProp3} above gives our desired inequality~\eqref{Eq:AGMProp1} for $k+1$. So to finish the proof, we have to prove the remaining terms in~\eqref{Eq:AGMProp3} are nonnegative.
We do so by applying two inequalities: We apply Lemma~\ref{Lem:HigherGrad} to the term $\langle \nabla f(y_{k+1}), x_{k+1}-y_{k+1} \rangle$. We also apply the classical Fenchel-Young inequality (e.g.,~\cite[Lemma~2]{Nesterov08}): $\langle s,h \rangle + \frac{1}{p} \|h\|^p \ge - \frac{p}{p-1} \|s\|^{\frac{p}{p-1}}_*$ with the choices $h = \epsilon^{-\frac{1}{p}} (x-z_k)$ and $s =  \epsilon^{\frac{1}{p}} Cp  (k+1)^{(p-1)} \nabla f(y_{k+1})$. Then from~\eqref{Eq:AGMProp3}, we obtain:
\begin{align}
\psi_{k+1}(x)
\,\ge\,
C(k+1)^{(p)}  \left[ f(y_{k+1})
+ \left( \frac{1}{4} - \frac{p^{\frac{2p-1}{p-1}}}{p-1} \, C^{\frac{1}{p-1}} \, \frac{\{(k+1)^{(p-1)}\}^{\frac{p}{p-1}}}{(k+1)^{(p)}} \right) \right] \, \epsilon^{\frac{1}{p-1}} \, \|\nabla f(y_{k+1})\|_*^{\frac{p}{p-1}}. \notag
\end{align}
Notice that $\{(k+1)^{(p-1)}\}^{\frac{p}{p-1}} \le (k+1)^{(p)}$. Then from the assumption $C \le (4p)^{-p}$, we see that the second term inside the parenthesis above is nonnegative. Hence we conclude the desired inequality $\psi_{k+1}(x) \ge C (k+1)^{(p)} f(y_{k+1})$, finishing the induction. 
\end{proof}

\medskip 
Finally, we can combine the result of Proposition~\ref{Prop:AGM1} with the basic estimate~\eqref{Eq:EstFunc2} at $x = x^*$, to conclude a convergence rate on $\aG_p$, which we summarize in the following theorem. 

\begin{theorem}\label{Thm:AGM}
If $f$ is $\frac{(p-1)!}{\epsilon}$-smooth of order $p-1$, $h$ is $\sigma$-uniformly convex of order $p$, and $C \le (4p)^{-p}$, then the $p$-th order accelerated gradient algorithm~\eqref{Eq:AGM} has convergence rate:
\begin{align}\label{Eq:AGMRate}
f(y_k) - f^* \,\le\, \frac{D_h(x^*,x_0)}{C \epsilon \sigma k^{(p)}} \,=\, O\left(\frac{1}{\epsilon k^p}\right).
\end{align}
\end{theorem}

The result above shows that by plugging in the $p$-th gradient method $\G_p$ into the accelerated algorithm~\eqref{Eq:AGM}, we boost its convergence rate from $O(1/\epsilon k^{p-1})$ to $O(1/\epsilon k^p)$. In particular, the family of accelerated gradient algorithms $\faG$ still maintains the nice sequential pattern of polynomial convergence rates, just like $\fG$.

\subsection{Continuous time limit}
\label{Sec:DiscToCts}

We show that the continuous time limit ($\epsilon \to 0$) of the accelerated algorithm $\aG_p$~\eqref{Eq:AGM} is a second order differential equation (with time scaling $\delta = \epsilon^p$). 
This is in contrast to the original $p$-th order gradient method $\G_p$~\eqref{Eq:HigherGrad}, which corresponds to the first-order rescaled gradient flow in continuous time (with time scaling $\delta = \epsilon^{p-1}$). Here we sketch the transition from discrete to continuous time, and in Section~\ref{Sec:NestFlowDisc} we will see the other view starting from the continuous time perspective.

We first note the following difference between the continuous time considerations of $\G_p$ and $\aG_p$. The $p$-th gradient method $\G_p$ is an algorithm that updates the sequence $x_k$ by:
\begin{align}\label{Eq:GradAlg}
x_{k+1} = \A_\epsilon(x_k)
\end{align}
where $\A_\epsilon \colon \X \to \X$ is the operator that returns the minimizer of the optimization problem~\eqref{Eq:HigherGrad}. The update~\eqref{Eq:GradAlg} above is equivalent to modeling the (discrete time) velocity $v_k = x_{k+1} - x_k$ as a function of the current position, $v_k = \A_\epsilon'(x_k)$ $(= \A_\epsilon(x_k) - x_k)$. As $\epsilon \to 0$, and by identifying $x_{k+1} = X_{t+\delta}$, $x_k = X_t$ with $\delta = \epsilon^{p-1}$, we recover a first order differential equation, the rescaled gradient flow~\eqref{Eq:RescGradFlow}. 

On the other hand, the accelerated gradient algorithm $\aG_p$ maintains three sequences:
\begin{align}\label{Eq:AccGradAlg}
(x_{k+1},y_{k+1},z_{k+1}) \,=\, \A_{k,\epsilon}(x_k,y_k,z_k)
\end{align}
where the operator $\A_{k,\epsilon}$ now changes over (discrete) time $k$. As in the preceding paragraph, as $\epsilon \to 0$ the update~\eqref{Eq:AccGradAlg} above gives rise to a system of first order differential equations in the variables $X_t,Z_t$ (and $Y_t = X_t$), which is equivalent to a second order equation in $X_t$. Moreover, since $\A_{k,\epsilon}$ depends with $k$, this second order equation has time-varying coefficients.

\paragraph{Derivation.}
We now analyze the updates~\eqref{Eq:AGM} in the limit $\epsilon \to 0$, starting with the $z$-update~\eqref{Eq:AGMz}.
\begin{itemize}
  \item {\bf $z$-update.} As noted in Section~\ref{Sec:HigherGradAlg}, we can write the $z$-update recursively as:
\begin{align}\label{Eq:AGMzw}
w_k \,=\, w_{k-1} - \epsilon Cp k^{(p-1)} \nabla f(y_k)
\end{align}
where $w_k = \nabla h(z_k)$, and here we set $\sigma$ (the uniform convexity constant of $h$) to be $1$ for simplicity.
Now we invoke the hypothesis that the sequences $y_k,w_k$ are discrete time snapshots of continuous time curves $Y_t,W_t$ (similarly for $x_k,z_k$ with respect to $X_t,Z_t$) at each time increment $\delta > 0$.
Specifically, we identify $y_k = Y_t$, $w_k = W_t$, and $w_{k-1} = W_{t-\delta}$, and use the linear approximation $W_{t-\delta} = W_t - \delta \dot W_t + o(\delta)$.
Under these identifications,~\eqref{Eq:AGMzw} becomes:
\begin{align}\label{Eq:AGMzw2}
\dot W_t = - \frac{\epsilon}{\delta} \, Cpk^{(p-1)} \nabla f(Y_t) + o(1).
\end{align}
With time increment $\delta$ we have the correspondence $t = \delta k$ or $k = t/\delta$, so $k^{(p-1)} \approx k^{p-1} = t^{p-1}/\delta^{p-1}$. Thus, on the right hand side of~\eqref{Eq:AGMzw2} we have the factor $\epsilon / \delta^p$. As $\epsilon \to 0$ and $\delta \to 0$, for the expression~\eqref{Eq:AGMzw2} above to have a meaningful content, we need to have the two variables to scale as $\epsilon = \delta^p$ (or $\epsilon = \Theta(\delta^p)$ in general). Under this scaling,~\eqref{Eq:AGMzw2} yields the differential equation for $W_t$:
\begin{align}\label{Eq:AGMzw3}
\dot W_t = - Cpt^{p-1} \nabla f(Y_t).
\end{align}
Since $w_k = \nabla h(z_k)$, $W_t = \nabla h(Z_t)$, we also have $\frac{d}{dt} \nabla h(Z_t) = -Ct^{p-1} \nabla f(Y_t)$, or equivalently:
\begin{align}\label{Eq:AGMzdot}
\dot Z_t = - Cpt^{p-1} \nabla^2 h(Z_t)^{-1}\, \nabla f(Y_t).
\end{align}
Notice,~\eqref{Eq:AGMzdot} is a first order equation in time since it involves the velocity $\dot Z_t$, but second order in space since it involves the Hessian $\nabla^2 h(Z_t)$.

  \item {\bf $y$-update.} Observe, the $y$-update~\eqref{Eq:AGMy} operates on a smaller time scale $\epsilon^{\frac{1}{p-1}}$. That is, from the first order optimality condition of $y_k$~\eqref{Eq:HigherGradLem1}, and the bound implied by the $\frac{(p-1)!}{\epsilon}$-Lipschitz assumption of $\nabla^{p-1} f$~\eqref{Eq:HigherGradLem2}, we have:\footnote{Lemma~\ref{Lem:HigherGrad} gives the reverse $\|x_k-y_k\| \ge \frac{1}{4} \epsilon^{\frac{1}{p-1}} \|\nabla f(y_k)\|_*^{\frac{1}{p-1}}$, so the bound is tight. Indeed, like in Section~\ref{Sec:EpsDelta}, the $y$-update~\eqref{Eq:AGMy} is a discretized rescaled gradient flow: $y_k = x_k - 2^{-\frac{1}{p-1}}  \epsilon^{\frac{1}{p-1}} \nabla f(x_k) / \|\nabla f(x_k)\|_*^{\frac{p-2}{p-1}} + o(\epsilon^{\frac{1}{p-1}})$.}
\begin{align}\label{Eq:xySmall}
\|x_k-y_k\| \,\le\, \epsilon^{\frac{1}{p-1}} \|\nabla f(y_k)\|_*^{\frac{1}{p-1}}.
\end{align}
With the identification $x_k = X_t$, $y_k = Y_t$, the bound~\eqref{Eq:xySmall} above shows the difference $X_t - Y_t = O(\epsilon^{\frac{1}{p-1}})$ is smaller than our time step $\delta = \epsilon^{\frac{1}{p}}$ (needed for~\eqref{Eq:AGMzdot}). Therefore:
\begin{align}\label{Eq:AGMXY}
X_t = Y_t.
\end{align}

  \item {\bf $x$-update.} The $x$-update~\eqref{Eq:AGMx} can be written as $x_{k+1}-x_k = \frac{p}{k+p}(z_k-x_k) + \frac{k}{k+p} (y_k-x_k)$. Identifying $x_{k+1} - x_k = \delta \dot X_t$, $z_k = Z_t$, and $\delta(k+p) = t$, as well as using the bound~\eqref{Eq:xySmall} with $\delta = \epsilon^{\frac{1}{p}}$, we get $\|\dot X_t - \frac{p}{t}(Z_t-X_t)\| \le \frac{1}{\delta} \cdot \epsilon^{\frac{1}{p-1}} \|\nabla f(Y_t)\|_*^{\frac{1}{p-1}} \to 0$, which means:
\begin{align}\label{Eq:AGMxdot}
\dot X_t = \frac{p}{t} (Z_t-X_t).
\end{align}
\end{itemize}


Thus, we conclude that the continuous time limit of the accelerated gradient method $\aG_p$~\eqref{Eq:AGM} is the system of first order differential equations~\eqref{Eq:AGMzdot},~\eqref{Eq:AGMXY},~\eqref{Eq:AGMxdot}. This system is equivalent to the following second order differential equation for $X_t$:
\begin{align}\label{Eq:AGMCont0}
\ddot X_t \,+\, \frac{p+1}{t} \dot X_t \,+\, Cp^2t^{p-2} \nabla^2 h\left(X_t + \frac{t}{p} \dot X_t\right)^{-1} \nabla f(X_t) = 0.
\end{align}
Note, even though the accelerated gradient methods $\aG_p$ use higher-order derivatives of $f$, in continuous time they are all second order differential equations~\eqref{Eq:AGMCont0}. This is in parallel to---but also in contrast from---how the $p$-th gradient method $\G_p$ is a $(p-1)$-st order algorithm (in space) but corresponds to a first order differential equation (in time), the rescaled gradient flow~\eqref{Eq:RescGradFlow}.

\section{Nesterov flow}
\label{Sec:NestFlow}

In this section we study the family of second order differential equation~\eqref{Eq:AGMCont0}, the {\em Nesterov flow}:
\begin{align}\label{Eq:NestFlow}
\ddot X_t + \frac{p+1}{t} \dot X_t + Cp^2t^{p-2} \nabla^2 h\left(X_t+\frac{t}{p}\dot X_t\right)^{-1} \nabla f(X_t) = 0
\end{align}
where $p > 0$ is not necessarily an integer, and $C > 0$ is a constant. Here we assume $f$ is convex and continuously differentiable, and $h$ is strictly convex so $\nabla^2 h$ is invertible. But we make no smoothness assumption on $f$ nor uniform convexity assumption on $h$.

Notice, we can equivalently write Nesterov flow~\eqref{Eq:NestFlow} as follows: 
\begin{align}\label{Eq:NestFlow2}
\frac{d}{dt} \nabla h\left(X_t+\frac{t}{p} \dot X_t \right) \;=\; \nabla^2 h\left(X_t+\frac{t}{p} \dot X_t \right) \left( \frac{p+1}{p} \dot X_t + \frac{t}{p} \ddot X_t \right) \;=\; -Cpt^{p-1} \nabla f(X_t).
\end{align}
As shown in Section~\ref{Sec:DiscToCts}, equation~\eqref{Eq:NestFlow2} is the continuous time limit of the $p$-th accelerated gradient method $\aG_p$~\eqref{Eq:AGM}, for integer $p \ge 2$.
The rewriting~\eqref{Eq:NestFlow2} is nicer than~\eqref{Eq:NestFlow} because it avoids the potential singularity problem at $t=0$ (i.e., no term $\frac{p+1}{t}$). 
Indeed, by setting $t \to 0$ in~\eqref{Eq:NestFlow2} we see that if $p > 1$, then $\dot X_0 = 0$. That is, any Nesterov curve $X_t$ that evolves following~\eqref{Eq:NestFlow},~\eqref{Eq:NestFlow2} must start from being at rest; it is the acceleration $\ddot X_t$ that drives the trajectory. 


\subsection{Convergence rate via energy functional}
\label{Sec:NestFlowRate}

Our interest in Nesterov flow~\eqref{Eq:NestFlow} stems from it being the continuous time limit of the $p$-th accelerated gradient method $\aG_p$, which has convergence rate $O(1/\epsilon k^p)$. 
(Theorem~\ref{Thm:AGM}).
We now show Nesterov flow preserves this convergence rate, without any additional assumptions on $f,h$ beyond convexity.
Define the {\em energy functional}:
\begin{align}\label{Eq:NestFlowE}
\E_t = C t^p (f(X_t)-f^*) + D_h\left(x^*, X_t+\frac{t}{p} \dot X_t \right)
\end{align}
where recall, $x^* = \arg\min_x f(x)$ and $f^* = f(x^*)$. 
It has time derivative:
\begin{align}\label{Eq:NestFlowEdot}
\dot \E_t
\,=\, Cp t^{p-1}\left(f(X_t) - f^* + \frac{t}{p} \langle \nabla f(X_t), \dot X_t \rangle \right) - \left\langle \frac{d}{dt} \nabla h\left(X_t+\frac{t}{p} \dot X_t \right), \, x^* - X_t - \frac{t}{p} \dot X_t \right \rangle.
\end{align}
If $X_t$ is governed by Nesterov flow~\eqref{Eq:NestFlow}, then $\dot \E_t$ above simplifies to:
\begin{align}\label{Eq:NestFlowEdot2}
\dot \E_t
\,=\, Cp t^{p-1} \big( f(X_t) - f^* + \langle \nabla f(X_t), x^*-X_t \rangle \big)
\stackrel{\eqref{Eq:Conv}}{\le} 0
\end{align}
where the last inequality follows from the convexity of $f$.
This means energy is decreasing over time: $\E_t \le \E_0 = D_h(x^*,X_0)$, for all $t \ge 0$. Since  $D_h(x^*,X_t+\frac{t}{p}\dot X_t) \ge 0$, we conclude that a curve $X_t$ governed by Nesterov flow~\eqref{Eq:NestFlow} has convergence guarantee:
\begin{align}\label{Eq:NestFlowRate}
f(X_t) - f^* \,\le\, \frac{D_h(x^*,X_0)}{Ct^p} \,=\, O\left(\frac{1}{t^p}\right)
\end{align}
which matches the $O(1/\epsilon k^p)$ convergence rate of $\aG_p$~\eqref{Eq:AGMRate} in discrete time, as claimed. But the bound~\eqref{Eq:NestFlowRate} holds for all $p > 0$, and only requires convexity of $f$ (in~\eqref{Eq:NestFlowEdot2}) and $h$ (so that $D_h \ge 0$).

We note,~\eqref{Eq:NestFlowE} is a generalization of the energy functional in~\cite{SuBoydCandes14}, who were the first to point out that accelerated gradient descent ($p=2$, Euclidean case) in continuous time corresponds to the second order equation $\ddot X_t + \frac{3}{t} \dot X_t + \nabla f(X_t) = 0$, and proved a matching $O(1/t^2)$ convergence rate~\cite[Theorem~3.2]{SuBoydCandes14}.
They also remarked on the significance of $3$ as being the smallest value of $r$ such that the modified equation 
$\ddot X_t + \frac{r}{t} \dot X_t + \nabla f(X_t) = 0$ 
has the same inverse quadratic $O((r-1)^2/t^2)$ convergence rate, and this guarantee breaks for $r < 3$, so there is a ``phase transition'' at $r=3$. 

We can explain the ``phase transition'' as follows, setting $r = p+1$. 
If we use $\frac{p+1}{t} \dot X_t$ as the velocity term in~\eqref{Eq:NestFlow}, then we should increase the weight of $\nabla f(X_t)$ to $\frac{p^2}{4}t^{p-2}$ in order to get the optimal convergence rate $O(1/t^p$)~\eqref{Eq:NestFlowRate}.
Indeed, we can generalize the energy functional~\eqref{Eq:NestFlowE} to $\E_t' = \rho_t(f(X_t)-f^*) + D_h(x^*,X_t+\frac{t}{p}\dot X_t)$ for any increasing $\rho_t > 0$ with $\rho_0 = 0$. By the same calculation~\eqref{Eq:NestFlowEdot}, if $\ddot X_t + \frac{p+1}{t} \dot X_t + \frac{p^2}{t^2} \rho_t \nabla f(X_t) = 0$, then $\dot{\E_t'} \le 0$ as long as $\rho_t \le Ct^p$, yielding convergence rate $O(1/\rho_t)$.
In particular, the equation $\ddot X_t + \frac{r}{t} \dot X_t + \nabla f(X_t) = 0$ from~\cite{SuBoydCandes14} is the case $\rho_t = t^2/(r-1)^2$ with convergence rate $O(1/\rho_t) = O((r-1)^2/t^2)$, consistent with their result. 
So indeed $3$ is special, as $3 = p+1$ when $p=2$.
Using $r = p+1 \ge 3$ requires weighting $\nabla f(X_t)$ by $\frac{p^2}{4}  t^{p-2}$, which is necessary for the $O(1/t^2)$ rate in continuous time. 

\subsection{Discretizing Nesterov flow}
\label{Sec:NestFlowDisc}

In this section we examine how to discretize Nesterov flow~\eqref{Eq:NestFlow} so as to preserve the convergence guarantee. Following the approach in Section~\ref{Sec:EpsDelta}, we choose to discretize the equivalent equation~\eqref{Eq:NestFlow2}, which can be written as a system of two first order equations:
\begin{subequations}\label{Eq:NestFlowXZ}
\begin{align}
Z_t &= X_t + \frac{t}{p} \dot X_t  \label{Eq:NestFlowXZa} \\
\frac{d}{dt} \nabla h(Z_t) &= -Cpt^{p-1} \nabla f(X_t).   \label{Eq:NestFlowXZb}
\end{align}
\end{subequations}
Now suppose we discretize $X_t,Z_t$ into sequences $x_k,z_k$ with time step $\delta > 0$, that is, if $x_k = X_t$ then $x_{k+1} = X_{t+\delta} = X_t+\delta \dot X_t$, and similarly for $z_k = Z_t$, $z_{k+1} = Z_{t+\delta} = Z_t+\delta \dot Z_t$. This means $k$ discrete iterations, each corresponding to a jump of length $\delta$, are equivalent to the elapse of $t = \delta k$ continuous time.

Under this identification,~\eqref{Eq:NestFlowXZa} becomes $z_k = x_k + \frac{t}{p} \frac{1}{\delta}(x_{k+1}-x_k)$, or equivalently:
\begin{align}\label{Eq:NestFlowAlg1}
x_{k+1} = \frac{p}{k} z_k + \frac{k-p}{k} x_k
\end{align}
which is the same as the $x$-update in $\aG_p$~\eqref{Eq:AGMx}, except here we use $x_k$ instead of $y_k$ (which is currently not in the algorithm). Moreover,~\eqref{Eq:NestFlowAlg1} uses convex weight $\frac{p}{k}$, but it is equivalent to the weight $\frac{p}{k+p} = \frac{p}{k} + o(\frac{1}{k})$ in~\eqref{Eq:AGMx}. Note, in~\eqref{Eq:NestFlowAlg1} there is no $\delta$.

Similarly,~\eqref{Eq:NestFlowXZb} becomes $\frac{1}{\delta}(\nabla h(z_{k+1}) - \nabla h(z_k)) = -Cp(\delta k)^{p-1} \nabla f(x_k)$, which we can recognize as the optimality condition of a mirror descent step:
\begin{align}\label{Eq:NestFlowAlg2}
z_{k+1} = \arg\min_z \,\left\{ Cp \, k^{p-1} \langle \nabla f(x_k), z \rangle + \frac{1}{\delta^p} D_h(z,z_k) \right\}
\end{align}
which is the same as he $z$-update in $\aG_p$~\eqref{Eq:AGMz2}, except here we use $\nabla f(x_k)$ instead of $\nabla f(y_k)$; moreover,~\eqref{Eq:NestFlowAlg2} uses $k^{p-1}$ instead of the equivalent weighting $(k+1)^{(p-1)} = \Theta(k^{p-1})$ in~\eqref{Eq:AGMz2}. We also see the scaling  $\epsilon = \delta^p$ of the step size $\epsilon$ in~\eqref{Eq:NestFlowAlg2} and the time step $\delta$ in the discretization.

In principle, the two updates~\eqref{Eq:NestFlowAlg1},~\eqref{Eq:NestFlowAlg2} define an algorithm that ``implements'' Nesterov flow~\eqref{Eq:NestFlowXZ} in discrete time. However, we also want a matching convergence rate guarantee $O(1/\epsilon k^p)$ (since $\epsilon k^p = (\delta k)^p = t^p$), and unfortunately that doesn't seem possible with only~\eqref{Eq:NestFlowAlg1},~\eqref{Eq:NestFlowAlg2}. We can try to follow the approach in the proof of Theorem~\ref{Thm:AGM}, and attempt to establish upper and lower estimates of the objective $f$ by the estimate function $\psi_k$. However, a key step in the proof is showing that the remainder of the expression~\eqref{Eq:AGMProp3} is nonnegative, for which we need the result~\eqref{Eq:HigherGradLem} in Lemma~\ref{Lem:HigherGrad} as well as the sequence $y_k$.
Thus, we can view the accelerated gradient algorithm $\aG_p$~\eqref{Eq:AGM} as a discretization of Nesterov flow~\eqref{Eq:NestFlowXZ}, with the introduction of an additional sequence $y_k$~\eqref{Eq:AGMy} whose purpose is to guarantee inequality~\eqref{Eq:HigherGradLem}, 
which implies the matching convergence rate $O(1/\epsilon k^p)$.

It is curious that we need the sequence $y_k$ satisfying~\eqref{Eq:HigherGradLem} to make the convergence proof work in discrete time. This $y_k$ differs from $x_k$ by a smaller time scale $\epsilon^{\frac{1}{p-1}} < \epsilon^{\frac{1}{p}} = \delta$~\eqref{Eq:xySmall}, so as $\delta \to 0$, $x_k$ and $y_k$ have the same continuous time limit $X_t=Y_t$. We also note that inequality~\eqref{Eq:HigherGradLem} is the only place where the $(p-1)$-st order smoothness of $f$ is needed (by $\G_p$, which is used by $y_k$~\eqref{Eq:AGMy}). It would be interesting to see whether it is possible to replace $\G_p$ in~\eqref{Eq:AGMy} by another algorithm that guarantees the same inequality~\eqref{Eq:HigherGradLem} under a weaker assumption.


\subsection{Interpretation as Euler-Lagrange equation}
\label{Sec:NestFlowEL}


Nesterov flow~\eqref{Eq:NestFlow} looks similar to the second order damped harmonic oscillator equation from physics, 
but with a subtle difference.
Recall, in classical mechanics we typically model friction as a velocity-dependent force. The equation of motion is then a second order equation involving both time derivatives of $X_t$, e.g., $\ddot X_t + 2\zeta \dot X_t + X_t = 0$ for damped harmonic oscillator with ``damping ratio'' $\zeta > 0$; whereas the ideal (frictionless) harmonic oscillator $\ddot X_t = -X_t$ involves no velocity term. The presence of $2\zeta\dot X_t$ in damped harmonic oscillator changes the nature of the system, from {\em conservative} (conserves {\em energy} $\E_t = \frac{1}{2} \|\dot X_t\|^2 + \frac{1}{2} \|X_t\|^2$, so ideal harmonic oscillator never stops) to {\em dissipative} (dissipates energy, the system stabilizes, and the curve $X_t$ converges).

Like damped harmonic oscillator, Nesterov flow~\eqref{Eq:NestFlow} is also a second order equation involving both time derivatives of $X_t$, although note that the velocity term in Nesterov flow has time-varying coefficient $\frac{p+1}{t}$. Nevertheless, we can still capture Nesterov flow under the same general framework of dissipative system, but with ``logarithmic damping'' (instead of linear); see Section~\ref{Sec:Lag}.

Concretely, we observe that Nesterov flow can be interpreted as the {\em Euler-Lagrange equation}:
\begin{align}\label{Eq:EL}
\frac{d}{dt} \left\{\frac{\partial}{\partial v} \L(X_t,\dot X_t,t) \right\} = \frac{\partial}{\partial x} \L(X_t,\dot X_t,t)
\end{align}
where $\L(x,v,t)$ is the {\em (Nesterov) Lagrangian} functional:
\begin{align}\label{Eq:NestLag}
\L(x,v,t) = pt^{p-1} D_h\left(x+\frac{t}{p}v, \, x\right) - Cp\,t^{2p-1} f(x)
\end{align}
defined for any point $x \in \X$, tangent vector $v \in \TxX$, and time $t \in \R$. In~\eqref{Eq:EL}, $\frac{\partial}{\partial v} \L(X_t,\dot X_t,t)$ is the partial derivative of $\L(x,v,t)$ with respect to $v$ evaluated at $(x,v,t) = (X_t,\dot X_t,t)$, and similarly for $\frac{\partial}{\partial x} \L(X_t,\dot X_t,t)$. For the Nesterov Lagrangian~\eqref{Eq:NestLag}, we can calculate these derivatives explicitly and verify that~\eqref{Eq:EL} above is indeed equivalent to Nesterov flow~\eqref{Eq:NestFlow}.

Recall, the Euler-Lagrange equation~\eqref{Eq:EL} is a necessary (and often sufficient) condition for $X_t$ to be a stationary point for the following variational problem, for any time $t_0 < t_1$~\cite[Theorem~1]{Gelfand}:
\begin{align}\label{Eq:Action}
\min \int_{t_0}^{t_1} \L(X_t,\dot X_t,t) \, dt
\end{align}
where the minimization is performed over all continuously differentiable curves $X_t$ with fixed endpoints $X_{t_0} = x_0, X_{t_1} = x_1 \in \X$.
The objective function in~\eqref{Eq:Action} is typically called the {\em action} $\cA(X)$, and the problem~\eqref{Eq:Action} is known as the 
{\em principle of least action}, which has played an important role as an equivalent reformulation of much of classical physics.\footnote{Chiefly among them Newton's law of motion $\ddot X_t = -\nabla f(X_t)$, which comes from  the {\em ideal Lagrangian} $\L_0(x,v,t) = \frac{1}{2}\|v\|^2 - f(x)$, where here $f$ is the ``potential'' function generating the ``force'' $F(x) = -\nabla f(x)$.}

Potential external connections aside, 
the significance of this observation is in making us aware of the rich structure in the family of Nesterov flows. Indeed, the interpretation of Nesterov flow as  Euler-Lagrange equation allows us access to the techniques and results from (for example) calculus of variations, which may not be applicable to the first order rescaled gradient flows.\footnote{Although, we can interpret rescaled gradient flow as the ``massless limit'' of a Lagrangian flow (Section~\ref{Sec:Mass}).}
And we believe that the discrete time accelerated algorithm $\aG_p$~\eqref{Eq:AGM}---whose convergence proof at first glance looks like an ``algebraic trick''---is actually exploiting this structure.

Furthermore, it turns out that the family of Nesterov Lagrangians~\eqref{Eq:NestLag} is particularly nice, and
in fact can be extended to a larger class of Lagrangians that preserves much of the nice properties.

\section{A Lagrangian view of acceleration}
\label{Sec:Lag}

We introduce the family of {\em Bregman Lagrangians}:
\begin{align}\label{Eq:BregLag}
\L_{\alpha,\beta,\gamma}(x,v,t) \,=\, e^{\gamma_t}\left(e^{2\alpha_t} D_h\left(x+e^{-\alpha_t} v, x\right) - e^{\beta_t} f(x) \right)
\end{align}
where $\alpha_t \in \R$ is the {\em scale function}, $\beta_t \in \R$ the {\em weight function}, and $\gamma_t \in \R$ the {\em damping function}.\, 
In the Euclidean case $h(x) = \frac{1}{2} \|x\|^2$ ($\ell_2$-norm), the Lagrangian~\eqref{Eq:BregLag} simplifies to (notice no $\alpha_t$):
\begin{align}\label{Eq:BregLagEuc}
\L_{\alpha,\beta,\gamma}(x,v,t) \,=\, e^{\gamma_t}\left(\frac{1}{2} \|v\|^2 - e^{\beta_t} f(x) \right).
\end{align}

\paragraph{Ideal scaling.}
In general, $\alpha_t,\beta_t,\gamma_t$ in~\eqref{Eq:BregLag} can be arbitrary, but we find there is an {\em ideal scaling} that is necessary for some results to hold (e.g.,~\eqref{Eq:IdeScaGam} simplifies~\eqref{Eq:ELBreg0} to~\eqref{Eq:ELBreg}):
\begin{subequations}\label{Eq:IdeSca}
\begin{align}
\dot \beta_t \,&=~\,2\dot \alpha_t + e^{\alpha_t}   \label{Eq:IdeScaBet} \\
\dot \gamma_t \,&=\, - \dot \alpha_t + e^{\alpha_t}  \label{Eq:IdeScaGam}.
\end{align}
\end{subequations}

\paragraph{Euler-Lagrange equation.}
For general functions $\alpha_t,\beta_t,\gamma_t$, the Euler-Lagrange equation~\eqref{Eq:EL} $\frac{d}{dt}\{\frac{\partial}{\partial v} \L(X_t,\dot X_t,t)\} = \frac{\partial}{\partial x} \L(X_t,\dot X_t,t)$ for the Bregman Lagrangian $\L \equiv \L_{\alpha,\beta,\gamma}$~\eqref{Eq:BregLag} is given by:
\begin{equation}\label{Eq:ELBreg0}
\begin{split}
\ddot X_t &+ \left(- \dot \alpha_t + e^{\alpha_t}\right) \dot X_t - e^{\beta_t} \, \nabla^2 h\left(X_t+e^{-\alpha_t} \dot X_t\right)^{-1} \nabla f(X_t) \\
&+ e^{\alpha_t}\left(\dot \gamma_t - \dot \alpha_t + e^{\alpha_t}\right) \nabla^2 h\left(X_t+e^{-\alpha_t} \dot X_t\right)^{-1} \left(\nabla h\left(X_t+e^{-\alpha_t} \dot X_t\right) - \nabla h(X_t)\right) = 0.
\end{split}
\end{equation}
If $\gamma_t$ satisfies the ideal scaling~\eqref{Eq:IdeScaGam}, then the Euler-Lagrange equation~\eqref{Eq:ELBreg0} simplifies to:
\begin{equation}\label{Eq:ELBreg}
\ddot X_t + \dot \gamma_t \, \dot X_t + e^{\beta_t} \, \nabla^2 h\left(X_t + e^{-\alpha_t} \dot X_t\right)^{-1} \nabla f(X_t) = 0
\end{equation}
which we call {\em Bregman flow}.
In the Euclidean case~\eqref{Eq:BregLagEuc}, it simplifies to $\ddot X_t + \dot \gamma_t \, \dot X_t + e^{\beta_t} \, \nabla f(X_t) = 0$.

\paragraph{Convergence rate via energy functional.}
Generalizing~\eqref{Eq:NestFlowE}, we define the {\em energy functional}:
\begin{align}\label{Eq:BregFlowE}
\E_t = e^{\beta_t-2\alpha_t} (f(X_t)-f^*) + D_h\left(x^*, X_t+e^{-\alpha_t} \dot X_t \right).
\end{align}
Assuming the ideal scaling $\dot \gamma_t = - \dot \alpha_t + e^{\alpha_t}$~\eqref{Eq:IdeSca}, $\E_t$ has time derivative:
\begin{equation}\notag
\begin{split}
\dot \E_t
&=\, (\dot \beta_t-2\dot \alpha_t) e^{\beta_t-2\alpha_t} (f(X_t)-f^*) \,+\, e^{\beta_t-2\alpha_t} \,\langle \nabla f(X_t), \dot X_t \rangle\\
&~~~~~~~~~ +\, e^{-\alpha_t} \, \left\langle \nabla^2 h\left(X_t+e^{-\alpha_t} \dot X_t\right) (\ddot X_t + \dot \gamma_t \dot X_t), \, X_t - x^* + e^{-\alpha_t} \dot X_t \right\rangle.
\end{split}
\end{equation}
If $X_t$ satisfies the Euler-Lagrange equation~\eqref{Eq:ELBreg}, then $\dot \E_t$ simplifies to:
\begin{align}\notag
\dot \E_t \,=\, (\dot \beta_t-2\dot \alpha_t) e^{\beta_t-2\alpha_t}(f(X_t)-f^*) - e^{\beta_t-\alpha_t} \langle \nabla f(X_t), X_t-x^* \rangle.
\end{align}
Now invoking the convexity of $f$~\eqref{Eq:Conv}, we can bound the expression above by:
\begin{align}\label{Eq:BregFlowEdot2}
\dot \E_t \,\le\, e^{\beta_t-2\alpha_t} \left(\dot \beta_t - 2\dot \alpha_t - e^{\alpha_t} \right)(f(X_t) - f^*).
\end{align}
Thus, we see that if $\dot \beta_t \le 2\dot \alpha_t + e^{\alpha_t}$, then $\dot \E_t \le 0$, which implies $\E_t \le \E_0$ for all $t \ge 0$. Since $h$ is convex, $D_h(x^*, X_t+e^{-\alpha_t} \dot X_t) \ge 0$, therefore we get the bound $e^{\beta_t-2\alpha_t}(f(X_t) - f^*) \le \E_0$. This gives a convergence rate $\rho_t = 2\alpha_t-\beta_t$, which we summarize in the following theorem.
Note, for any $\alpha_t$, the optimal choice of $\beta_t$ in~\eqref{Eq:BregRate} below is given by the ideal scaling~\eqref{Eq:IdeScaBet}, which yields rate $\rho_t = 2\alpha_t-\beta_t = \int_0^t e^{\alpha_s} ds$.

\begin{theorem}\label{Thm:BregRate}
If $f$ and $h$ are convex and the ideal scaling~\eqref{Eq:IdeScaGam} holds, then 
for any $\alpha_t,\beta_t$ satisfying $\dot \beta_t \le 2\dot \alpha_t + e^{\alpha_t}$, the curve $X_t$ governed by the Euler-Lagrange equation~\eqref{Eq:ELBreg} has convergence rate $\rho_t = 2\alpha_t-\beta_t$:
\begin{align}\label{Eq:BregRate}
f(X_t) - f^* \,\le\, \frac{\E_0}{e^{\beta_t-2\alpha_t}} \,=\, O\left(e^{2\alpha_t-\beta_t} \right).
\end{align}
\end{theorem}

\paragraph{Example: Nesterov Lagrangian.}
As our motivating example, Nesterov Lagrangian~\eqref{Eq:NestLag} is a special case of the Bregman Lagrangian~\eqref{Eq:BregLag} when we choose:
\begin{subequations}\label{Eq:NestParam}
\begin{align}
\alpha_t \,&=\, -\log t + \log p   \label{Eq:NestParamA} \\
\beta_t \,&=\, (p-2)\log t + 2\log p + \log C    \label{Eq:NestParamB} \\
\gamma_t \,&=\, (p+1) \log t - \log p   \label{Eq:NestParamC} \\
\rho_t \,&=\, p \log t   \label{Eq:NestParamR}
\end{align}
\end{subequations}
which satisfy the ideal scaling~\eqref{Eq:IdeSca}\footnote{Note,~\eqref{Eq:IdeSca} only determines $\beta_t,\gamma_t$ up to constant terms, but we will see in Section~\ref{Sec:TimeBreg} that~\eqref{Eq:NestParam} is the proper choice of constants from the perspective of time dilation.}
for any $p > 0$.
With these parameter choices, we can verify that Bregman Lagrangian~\eqref{Eq:BregLag} reduces to Nesterov Lagrangian~\eqref{Eq:NestLag}, Bregman flow~\eqref{Eq:ELBreg} reduces to Nesterov flow~\eqref{Eq:NestFlow}, and the convergence rate~\eqref{Eq:BregRate} recovers our earlier result~\eqref{Eq:NestFlowRate}.

However, the Bregman Lagrangian family~\eqref{Eq:BregLag} is much more general, and we wish to study its more general properties (which we can then specialize to any subfamily, including Nesterov~\eqref{Eq:NestLag}).
To expand the repertoire of Bregman Lagrangians, in Section~\ref{Sec:ExpConv} we study the family of constant $\alpha_t = \log c$ with rate $\rho_t = ct$, and show connections to the restart scheme proposed by Nesterov~\cite{Nesterov08} to obtain linear convergence in discrete time under uniform convexity assumption.

In the rest of this section, we discuss the interpretation of Bregman Lagrangian as approximating the ``true'' momentum method (Hessian Lagrangian); we will also see how to interpret rescaled gradient flow as the massless limit of a (modified) Lagrangian flow.

\subsection{Bregman Lagrangian as an approximation of Hessian Lagrangian}
\label{Sec:HessLag}

We define the family of {\em Hessian Lagrangian}:
\begin{align}\label{Eq:HessLag}
\L_{\beta,\gamma}(x,v,t) \,=\, e^{\gamma_t} \left(\frac{1}{2} \|v\|^2_{h(x)} - e^{\beta_t} f(x) \right)
\end{align}
where as before $\beta_t \in \R$ is the weight function and $\gamma_t \in \R$ the damping function. The Hessian Lagrangian~\eqref{Eq:HessLag} is a damped and weighted version of the ideal Lagrangian (with Hessian metric) $\L_0(x,v,t) = \frac{1}{2}\|v\|_{h(x)}^2 - f(x)$.

The Bregman divergence, being a first order approximation error, can be seen as approximating the squared Hessian norm:
\begin{align}\label{Eq:BregHess}
e^{2\alpha} D_h\left(x+e^{-\alpha} v, x\right) \,\approx\, \frac{1}{2} \|v\|^2_{h(x)}
\end{align}
for any $x \in \X$, $v \in \TxX$, and $\alpha \in \R$ (such that $x+e^{-\alpha} v \in \X$).
Therefore, we can interpret Bregman Lagrangian~\eqref{Eq:BregLag} as approximating the Hessian Lagrangian~\eqref{Eq:HessLag}: $\L_{\alpha,\beta,\gamma} \approx \L_{\beta,\gamma}$, for all $\alpha_t,\beta_t,\gamma_t$.
Note in the Euclidean case~\eqref{Eq:BregHess} is an equality so Bregman and Hessian Lagrangians coincide.

The Euler-Lagrange equation for the Hessian Lagrangian~\eqref{Eq:HessLag} is given by:
\begin{align}\label{Eq:ELHess}
\frac{1}{2} \nabla^3 h(X_t) \, \dot X_t \, \dot X_t \,+\, \nabla^2 h(X_t) \left(\ddot X_t + \dot \gamma_t \, \dot X_t \right) + e^{\beta_t}\nabla f(X_t) = 0
\end{align}
where the third order derivative $\nabla^3 h$ comes from being the derivative of the metric tensor $\nabla^2 h$. Notice if we remove the first term from~\eqref{Eq:ELHess}, then we recover the Euler-Lagrange equation~\eqref{Eq:ELBreg} for the Bregman Lagrangian in the case $\alpha_t = \infty$ (which is the ideal case since by L'H\^{o}pital's rule, $\lim_{\alpha \to \infty} e^{2\alpha}D_h(x+e^{-\alpha} v, x) = \frac{1}{2} \|v\|^2_{h(x)}$).
Thus, Bregman flow~\eqref{Eq:ELBreg} can be interpreted as an approximation to the {\em Hessian flow}~\eqref{Eq:ELHess} that removes the $\nabla^3 h$ term (and compensates by using $\nabla^2 h(X_t+e^{-\alpha_t} \dot X_t)$).
Note, Bregman flow can be equivalently written as~\eqref{Eq:NestFlowXZ}, which can then be discretized using mirror descent to yield an algorithm~\eqref{Eq:AGM}, which does not require $\nabla^2 h$ but only $\nabla h$.
Therefore, this offers an interpretation of Nesterov's acceleration technique as a clever approximate discretization of the Hessian Lagrangian, which reduces the complexity of the required computation from $\nabla^3 h$ in~\eqref{Eq:ELHess} down to $\nabla h$ in~\eqref{Eq:AGM}.

However, it is actually still unclear why the Hessian Lagrangian is the right thing to approximate. 
For example, we do not have any convergence guarantee on the Hessian flow. The convergence rate $\rho_t = \int_0^t e^{\alpha_s} ds$~\eqref{Eq:BregRate} for Bregman Lagrangian tends to $\infty$ as $\alpha \to \infty$. This suggests that in the ideal limit $\alpha_t = \infty$, Hessian flow has instantaneous convergence ($f(X_t) - f^* \le 0$ for any $t > 0$), although note also that under the ideal scaling~\eqref{Eq:IdeSca}, $\beta, \gamma \to \infty$ as $\alpha \to \infty$ so this limit is not well defined.

Let us take a step back and notice, in approximating Hessian norm by Bregman divergence~\eqref{Eq:BregHess} we have introduced a scale variable $\alpha \in \R$, which provides the conversion factor between the scales of the point $x \in \X$ and the tangent vector $v \in \TxX$. 
Ordinarily, we treat $v$ as operating at a small (infinitesimal) scale $\varepsilon > 0$ where linear approximation holds, e.g., $f(x+\varepsilon v) = f(x) + \varepsilon \langle \nabla f(x), v \rangle$. But in practice, how should we choose $\alpha$? As noted, the ideal is $\alpha = \infty$, in which case Bregman Lagrangian reduces to Hessian Lagrangian. However, as soon as $\alpha = \log(1/\varepsilon) < \infty$, there is the ideal scaling~\eqref{Eq:IdeSca} in Bregman Lagrangian that binds $\alpha,\beta,\gamma$ together, and renders the limit $\alpha \to \infty$ nonsensical. But in return, the ideal scaling gives us convergence rate $\rho_t = \int_0^t e^{\alpha_s} ds$, 
which is better for larger $\alpha_t$.
For example, Nesterov Lagrangian~\eqref{Eq:NestLag} uses logarithmic $\alpha_t = -\log t + \log p$, which yields $p$-sublinear rate $\rho_t = p\log t$. The exponential analog of Nesterov in Section~\ref{Sec:expLag} uses constant $\alpha_t = \log c$, which yields linear rate $\rho_t = ct$.

\subsection{Rescaled gradient flow as massless limit of Lagrangian flow}
\label{Sec:Mass}

We define the {\em $p$-th power Lagrangian}, $p > 0$:
\begin{align}\label{Eq:pLag}
\L(x,v,t) \,=\, e^{t/m} \left(\frac{m}{p} \|v\|^p - f(x)\right)
\end{align}
where $m > 0$ is the {\em mass} of our (fictitious) particle. In Section~\ref{Sec:Lag} we implicitly set $m=1$, but here we are interested in the limiting behavior $m \to 0$, with $p$ fixed.

The Euler-Lagrange equation $\frac{d}{dt} \{ \frac{\partial}{\partial v} \L(X_t,\dot X_t,t)\} = \frac{\partial}{\partial x} \L(X_t,\dot X_t,t)$ for the Lagrangian~\eqref{Eq:pLag} is:
\begin{align}\label{Eq:ELp}
\|\dot X_t\|^{p-2} \big(m\ddot X_t + \dot X_t \big) + m(p-2) \|\dot X_t\|^{p-4} \langle \ddot X_t,\dot X_t \rangle \, \dot X_t + \nabla f(X_t) = 0.
\end{align}
So as $m \to 0$, the Euler-Lagrange equation~\eqref{Eq:ELp} converges to the rescaled gradient flow~\eqref{Eq:RescGradFlow2}:
\begin{align}\notag
\|\dot X_t\|^{p-2} \dot X_t + \nabla f(X_t) = 0.
\end{align}
This gives an interpretation of rescaled gradient flow---which is a first order equation---as the massless limit of the Lagrangian flow~\eqref{Eq:ELp}, which is a second order equation. However, notice that this massless limit also corresponds to infinite momentum: $e^{t/m} m \|\dot X_t\|^{p-2} \dot X_t \to \infty$ as $m \to 0$. This means as $m \to 0$, our particle (whose evolution is governed by~\eqref{Eq:ELp}) becomes infinitely massive. In the limit $m = 0$ (rescaled gradient flow) there is no oscillation; the particle just rolls downhill (with infinite ``friction'') and stops at the minimum $x^*$ as soon as the force $-\nabla f$ vanishes.

Note, this may run contrary to the idea that adding an acceleration/momentum term amounts to preventing oscillation, hence the faster convergence. What our interpretation suggests is the opposite: The first order rescaled gradient flow is the case of infinite momentum and no oscillation (rather than no momentum and big oscillation), and the effect of moving to the second order Lagrangian flow is to unwind the curve (to finite momentum) where it travels faster, but there is now oscillation. For example, the case $p=2$ in~\eqref{Eq:ELp} is the damped harmonic oscillator (when $f(x) = \frac{1}{2}\|x\|^2$). This point of view is also consistent with the work~\cite{ODonoghue15}, who addressed this oscillation issue by a restart scheme (cf.~Section~\ref{Sec:NestRest}).
See also Appendix~\ref{Sec:NaturalMasslessLim} for an interpretation of natural gradient flow as the massless limit of Hessian Lagrangian flow.

\section{Linear convergence rate via uniform convexity}
\label{Sec:ExpConv}

We study the exponential analog of Nesterov Lagrangian, which uses constant scale function and has linear convergence rate. We also show how to extend Nesterov's restart scheme~\cite{Nesterov08} in discrete time to get linear convergence rate when the objective function is uniformly convex. 

\subsection{Exponential Nesterov Lagrangian family}
\label{Sec:expLag}

We define the {\em exp-Nesterov Lagrangian}, for any $c > 0$:
\begin{align}\label{Eq:expLag}
\L(x,v,t) \,=\, ce^{ct} \left(D_h\left(x+\frac{1}{c} v, \, x \right) - e^{ct} f(x) \right).
\end{align}
This is the Bregman Lagrangian~\eqref{Eq:BregLag} with the following choices, which satisfy the ideal scaling~\eqref{Eq:IdeSca}:
\begin{subequations}\label{Eq:expParam}
\vspace{-15pt}
\begin{align}
\alpha_t \,&=\, \log c \label{Eq:expParamA}  \\
\beta_t \,&=\, ct + 2\log c \label{Eq:expParamB} \\
\gamma_t \,&=\, ct-\log c \label{Eq:expParamC} \\
\rho_t \,&=\, ct.  \label{Eq:expParamR}
\end{align}
\end{subequations}
The Euler-Lagrange equation for the Lagrangian~\eqref{Eq:expLag} is given by:
\begin{equation}\label{Eq:ELexp}
\ddot X_t \,+\, c \dot X_t \,+\, c^2e^{ct} \,\nabla^2 h\left(X_t+\frac{1}{c}\dot X_t\right)^{-1} \nabla f(X_t) = 0.
\end{equation}
By Theorem~\ref{Thm:BregRate}, the {\em exp-Nesterov flow} $X_t$~\eqref{Eq:ELexp} has linear convergence rate $\rho_t = ct$ for convex $f$:
\begin{align}\label{Eq:expRate}
f(X_t) - f^* \,\le\, \frac{f(X_0)-f^* + D_h(x^*, X_0 + \frac{1}{c} \dot X_t)}{e^{ct}} \,=\, O(e^{-ct}).
\end{align}
Thus, whereas Nesterov flow \eqref{Eq:NestFlow} has sublinear rate, exp-Nesterov flow \eqref{Eq:ELexp} has linear rate.

Let us examine how to discretize exp-Nesterov flow. As in Section~\ref{Sec:NestFlowDisc}, we first write~\eqref{Eq:ELexp} as:
\begin{subequations}\label{Eq:expFlowXZ}
\begin{align}
Z_t &= X_t + \frac{1}{c} \dot X_t  \label{Eq:expFlowXZa} \\
\frac{d}{dt} \nabla h(Z_t) &= -ce^{ct} \nabla f(X_t).   \label{Eq:expFlowXZb}
\end{align}
\end{subequations}
We discretize $X_t,Z_t$ into sequences $x_k,z_k$ with time step $\delta > 0$ as before, so $t = \delta k$. Using mirror descent to implement~\eqref{Eq:expFlowXZb}, we obtain discrete time equations similar to~\eqref{Eq:AGMx},~\eqref{Eq:AGMz2}:
\begin{subequations}\label{Eq:expFlowAlg}
\begin{align}
z_k \,&=\, \arg\min_z \,\left\{ ce^{c\delta k} \langle \nabla f(x_k), z \rangle + \frac{1}{\delta} 
D_h(z,z_{k-1}) \right\} \label{Eq:expFlowAlgb} \\
x_{k+1} \,&=\, c\delta z_k + (1-c\delta) x_k.   \label{Eq:expFlowAlga}
\end{align}
\end{subequations}
Note, the weight in~\eqref{Eq:expFlowAlga} is independent of time, but depends on $\delta$; and~\eqref{Eq:expFlowAlgb} suggests the step size $\epsilon = \delta$ in the algorithm.
If our analogy between continuous and discrete time convergence holds, then given the $O(e^{-ct})$ convergence rate~\eqref{Eq:expRate} for $X_t$, we expect a matching $O(\frac{1}{\epsilon} e^{-ck})$ convergence rate in discrete time.

However, it is not clear how to get that with only~\eqref{Eq:expFlowAlg}. If we try to adapt the proof of Theorem~\ref{Thm:AGM} (with the ideal choice $p=\infty$), then we find that we need to introduce a sequence $y_k$ satisfying the following analog of Lemma~\ref{Lem:HigherGrad}, in order to conclude a convergence rate $O(\delta e^{-c\delta k})$:
\begin{align}\label{Eq:expNeeded}
\langle f(y_k), x_k-y_k \rangle \,\ge\, \delta e^{-c\delta k} \|\nabla f(y_k)\|_*.
\end{align}
Notice, the rates above are consistent if we set $\epsilon = \delta = 1$. But the condition~\eqref{Eq:expNeeded} means we need to make a constant improvement in each iteration from $x_k$ to $y_k$, although we are also free to be creative with how to construct $y_k$ and impose any assumption on $f$.

\subsection{Restart scheme for uniformly convex objective function}
\label{Sec:NestRest}

We present a generalization of Nesterov's restart scheme~\cite{Nesterov08} that obtains a linear convergence rate when the objective function is both smooth and uniformly convex. This in a sense can be seen as a discrete time version of exp-Nesterov flow~\eqref{Eq:expFlowXZ}, which implements the exponential weighting and the improvement requirement~\eqref{Eq:expNeeded} by {\em running the accelerated gradient algorithm $\aG_p$~\eqref{Eq:AGM} for some amount of time}, within each iteration.

\paragraph{Linear convergence rate for $\G_p$.}
Following~\cite[Section~5]{Nesterov08}, we first show that the $p$-th gradient method $\G_p$ has linear convergence rate if $f$ is uniformly convex. Concretely, consider the following $\G_p$ with larger regularization (like~\eqref{Eq:AGMy}), for $p \ge 2$:
\begin{align}\label{Eq:expGp}
x_{k+1} \,=\, \arg\min_x \left\{ f_{p-1}(x;x_k) + \frac{2}{\epsilon} \cdot \frac{1}{p} \|x-x_k\|^p \right\}
\end{align}
where $\epsilon > 0$ here is a fixed step size. By Lemma~\ref{Lem:HigherGrad}, we know that if $f$ is $\frac{(p-1)!}{\epsilon}$-smooth of order $p-1$, then $\langle \nabla f(x_{k+1}), x_k-x_{k+1} \rangle \geq \frac{1}{4} \epsilon^{\frac{1}{p-1}} \|\nabla f(x_{k+1})\|_*^{\frac{p}{p-1}}$.
If $f$ is $\sigma$-uniformly convex of order $p$, then we also have~\cite[Lemma~3]{Nesterov08} $\|\nabla f(x_{k+1})\|_*^{\frac{p}{p-1}} \ge \frac{p}{p-1} \sigma^{\frac{1}{p-1}} (f(x_{k+1})-f^*)$.
Moreover, if $f$ is convex~\eqref{Eq:Conv}, then:
\begin{align}\label{Eq:expGp2}
f(x_{k+1}) - f^* \,\le\, \frac{f(x_k)-f^*}{1+\frac{1}{4} (\epsilon \sigma)^{\frac{1}{p-1}}}
\,\le\, \frac{f(x_1)-f^*}{\big(1+\frac{1}{4} (\epsilon \sigma)^{\frac{1}{p-1}}\big)^k}.
\end{align}
If the {\em (inverse) condition number} $\kappa = \epsilon \sigma$ is small, then $1+\frac{1}{4} (\epsilon \sigma)^{\frac{1}{p-1}} \approx e^{\kappa^{\frac{1}{p-1}}/4}$.
And again by the smoothness of $f$, we have $f(x_1) \le \min_x \{f_{p-1}(x;x_0) + \frac{2}{\epsilon p} \|x-x_0\|^p\} \le f^* + \frac{3}{\epsilon p} \|x_0-x^*\|^p$.
Therefore,~\eqref{Eq:expGp2} yields the convergence rate $\rho_k = ck$, $c = \frac{1}{4} \kappa^{\frac{1}{p-1}}$, for the $p$-th gradient method $\G_p$~\eqref{Eq:expGp}, generalizing the result of~\cite[(5.6)]{Nesterov08}:
\begin{align}\label{Eq:expGpRate}
f(x_{k+1})-f^* \,\le\, \frac{3\|x_0-x^*\|^p}{\epsilon p \big(1+\frac{1}{4}\kappa^{\frac{1}{p-1}}\big)^k}
\,=\, O\left(\frac{1}{\epsilon}  \, e^{-\kappa^{\frac{1}{p-1}} k/4}\right)
\end{align}
which matches the desired convergence rate $O(\frac{1}{\epsilon} e^{-ck})$ discussed in Section~\ref{Sec:expLag}. 

\paragraph{Improved linear convergence rate for $\aG_p$ with restart.}
We now show that a variant of the accelerated gradient method $\aG_p$ attains a better linear convergence rate than $\G_p$~\eqref{Eq:expGp}. 
Specifically, consider the following restart scheme, generalizing~\cite[(5.7)]{Nesterov08}:
\begin{align}\label{Eq:expRest}
x_{(k+1)m} \,=\, \big(\text{the output $\hat y_m$ of running $\aG_p$~\eqref{Eq:AGM} for $m$ iterations with input $\hat x_0 = x_{km}$}\big)
\end{align}
where $m = 24p/\kappa^{\frac{1}{p}}$, and $\kappa = \epsilon \sigma$ as before is the inverse condition number of $f$. Here we assume we replace the Bregman divergence in the $z$-update~\eqref{Eq:AGMz} by $d_p(z) = \frac{1}{p} \|z-x_0\|^p$, rescaled by its uniform convexity constant $2^{-p+2}$~\eqref{Eq:Dist}. 
The proof of Theorem~\ref{Thm:AGM} still holds in this case, and for concreteness we choose $C = (4p)^{-p}$.

Then, since $f$ is $\sigma$-uniformly convex of order $p$~\eqref{Eq:UC}, and by the bound~\eqref{Eq:AGMRate} from Theorem~\ref{Thm:AGM}:
\begin{align}\label{Eq:expRest2}
\frac{\sigma}{p} \|x_{(k+1)m}-x^*\|^p
\stackrel{\eqref{Eq:UC}}{\le} f(x_{(k+1)m})-f^*
\stackrel{\eqref{Eq:AGMRate}}{\le} \frac{(4p)^p \, 2^{p-2} \|x_{km}-x^*\|^p}{\epsilon m^{(p)}}
\,\le\, \frac{\sigma}{pe} \|x_{km}-x^*\|^p
\end{align}
where the last inequality follows from our choice of $m$. Iterating~\eqref{Eq:expRest2} and rescaling the index $k \equiv \frac{k}{m}$, we obtain $\|x_k-x^*\|^p \le e^{-k/m}\|x_0-x^*\|^p$.
To convert this into a bound on the function value, we use the smoothness of $f$. 
Let $y_k$ be the output of $\G_p$~\eqref{Eq:expGp} with input $x_k$. As noted before, if $f$ is $\frac{(p-1)!}{\epsilon}$-smooth of order $p-1$, then $f(y_k) - f^* \le \frac{3}{\epsilon p} \|x_k-x^*\|^p$. Therefore, we conclude that:
\begin{align}\label{Eq:expRestRate}
f(y_k) - f^* \,\le\, \frac{3\|x_0-x^*\|^p}{\epsilon p \, e^{k/m}} \,=\, O\left(\frac{1}{\epsilon}  \, e^{-\kappa^{\frac{1}{p}} k/24 p}\right)
\end{align}
which matches the convergence rate $O(\frac{1}{\epsilon} e^{-ck})$ as discussed in Section~\ref{Sec:expLag} with $c = \frac{1}{24 p} \kappa^{\frac{1}{p}}$. Note, this linear rate $\rho_t = c\kappa^{\frac{1}{p}}$ has better dependence for small $\kappa = \epsilon \sigma$ than~\eqref{Eq:expGpRate}, generalizing the conclusion of~\cite[(5.8)]{Nesterov08}.
However, the link to continuous time is not as clear as that of the Nesterov family. 


\section{Time dilation: Faster convergence by speeding up time}
\label{Sec:Time}

In this section we introduce the idea of time dilation, and show that a large family of Bregman Lagrangians (which include Nesterov and exp-Nesterov) can be interpreted as the result of speeding up {\em any} single curve.

\subsection{Time dilation}

We recall the argument in Section~\ref{Sec:Intro} that optimization in continuous time is easy because we can get arbitrarily fast convergence. The idea is that once we have a curve $X_t$ that converges at some rate $\rho(t)$, then we can speed it up to $Y_t = X_{\tau(t)}$ with improved convergence rate $\rho(\tau(t)) > \tau(t)$.

Here $\tau = \tau(t) \in \R$ is a {\em time dilation}: A smooth, strictly increasing (hence invertible) function defined on the time domain $\R$ (or a subset of it), whose inverse $\tau^{-1}$ is also smooth.
The set $\sT$ of time dilations forms a group under function composition. The identity element is the {\em identity time}:
\begin{align}\label{Eq:tid}
\tid(t) = t~~~~ \forall \, t \in \R
\end{align}
which is the default time dilation we use, so normally time flows at unit speed: $d\tid/dt = 1$.

Traversing a curve $X_t$ at another time speed $Y_t = X_{\tau(t)}$ is equivalent to replacing the default time dilation $\tid$ by $\tau \in \sT$, so now time flows at speed $d\tau/dt = \dot \tau(t)$. We say that $\tau$ is {\em faster} than $\tid$ if $\dot \tau(t) \ge 1$, in which case shifting from $\tid$ to $\tau$ amounts to speeding up time, and the convergence rate $\rho(\tau(t))$ of $Y_t$ is larger than the original rate $\rho(t)$ of $X_t$ (if $\rho(t)$ is increasing).\footnote{But for convenience, we refer to $Y_t = X_{\tau(t)}$ as the {\em sped-up version} of $X_t$ regardless of whether $\tau$ is faster than $\tid$.}

However, how meaningful is this idea? Recall, our main interest is in understanding the parallel behavior between continuous and discrete time optimization, so even if we have a fast rate in continuous time, it may not be of interest to us if we don't know how to implement it in discrete time (with matching convergence rate).

Recall also from Section~\ref{Sec:Intro} the example of the gradient flow $\dot X_t = -\nabla f(X_t)$, which has convergence rate $\rho_t = -\log t$. The sped-up version $Y_t = X_{\tau(t)}$ satisfies $\dot Y_t = -\dot \tau(t) \nabla f(Y_t)$, but is is not a gradient flow anymore (not of the form $\dot Y_t = -\nabla \tilde f(Y_t)$ for some function $\tilde f$ that does not explicitly depend on time). Similarly, speeding up rescaled gradient flow~\eqref{Eq:RescGradFlow} yields a curve that is not in the rescaled gradient flow family.
While these properties inhibit our understanding of how these curves relate to each other in continuous time, it turns out Bregman Lagrangian flows~\eqref{Eq:ELBreg} have nice properties under time dilation, which we explore next.


\subsection{Bregman Lagrangian family under time dilation}
\label{Sec:TimeBreg}

We show that Bregman Lagrangian family is closed under the action of the time dilation group $\sT$. 
Note, in general it holds that speeding up a Lagrangian curve (i.e., Euler-Lagrange curve $X_t$ for a Lagrangian $\L$) results in another Lagrangian curve, so the space of general Lagrangian curves is closed under time dilation.
Moreover, the Bregman Lagrangians~\eqref{Eq:BregLag} form a special subfamily of the Lagrangian space, and we can characterize precisely how they transform under time dilation.

We begin by noting that the sped-up curve $Y_t = X_{\tau(t)}$ has time derivatives:
\begin{subequations}\label{Eq:TimeYdot}
\begin{align}
\dot Y_t &= \dot \tau(t) \, \dot X_{\tau(t)}   \label{Eq:TmYdota} \\
\ddot Y_t &= \ddot \tau(t) \, \dot X_{\tau(t)} + \dot \tau(t)^2 \ddot X_{\tau(t)}.   \label{Eq:TmYdotb}
\end{align}
\end{subequations}
Thus, if $X_t$ satisfies the Euler-Lagrange equation~\eqref{Eq:ELBreg} for a Bregman Lagrangian $\L = \L_{\alpha,\beta,\gamma}$~\eqref{Eq:BregLag}, then $Y_t = X_{\tau(t)}$ also satisfies the Euler-Lagrange equation for the modified Lagrangian $\L^{(\tau)} = \L_{\alpha^{(\tau)},\wt\beta^{(\tau)},\wt\gamma^{(\tau)}}$, where the parameters $\alpha,\beta,\gamma$ are transformed to $\alpha^{(\tau)},\beta^{(\tau)},\gamma^{(\tau)}$:
\begin{subequations}\label{Eq:TimeABC}
\begin{align}
\alpha_t^{(\tau)} \,&=\, \alpha_{\tau(t)} + \log \dot \tau(t)   \label{Eq:TimeA}  \\
\beta_t^{(\tau)} \,&=\, \beta_{\tau(t)} + 2\log \dot \tau(t)   \label{Eq:TimeB}  \\
\gamma_t^{(\tau)} \,&=\, \gamma_{\tau(t)} - \log \dot \tau(t).   \label{Eq:TimeC}
\end{align}
\end{subequations}
This means each $\tau \in \sT$ induces a map $\L \mapsto \L^{(\tau)}$ on the space of general Bregman Lagrangians:
\begin{align}\label{Eq:BregSpac}
\sL = \{\L_{\alpha,\beta,\gamma} \,\colon\, \alpha_t, \beta_t, \gamma_t \in \R \}.
\end{align}
Furthermore, by chain rule, we see that the transformation~\eqref{Eq:TimeABC} satisfies the composition property $(\alpha^{(\tau)})^{(\theta)} = \alpha^{(\tau \, \circ \, \theta)}$ for all $\tau,\theta \in \sT$, and similarly for $\beta$, $\gamma$ (that is, speeding up by $\tau$ and then $\theta$ is equivalent to speeding up once by $\tau \circ \theta$).
This lifts to the Lagrangian level: $(\L^{(\tau)})^{(\theta)} \,=\, \L^{(\tau \,\circ\, \theta)}$. Formally, this means the mapping $\L \mapsto \L^{(\tau)}$ is a (right) {\em group action} of $\sT$ on $\sL$, namely, a group homomorphism from $\sT$ to the permutation group of $\sL$.

This conclusion extends to Hessian Lagrangian, which is the $\alpha \to \infty$ limit of Bregman Lagrangian; thus, $\sT$ also acts on the space of Hessian Lagrangians with the same transformation rules~\eqref{Eq:TimeB},~\eqref{Eq:TimeC}.\footnote{Note, we can also show that $\sT$ acts on the family of $p$-th power Lagrangian $\L_{\beta,\gamma}(x,v,t) = e^{\gamma_t}(\frac{1}{p}\|v\|^p - e^{\beta_t} f(x))$ (\eqref{Eq:pLag} is the case $\gamma_t = t/m$ and $\beta_t = 0$), where now $\beta$ transforms as $\beta_t^{(\tau)} = \beta_{\tau(t)} + p\log \dot \tau(t)$ (\eqref{Eq:TimeB} is case $p=2$).}

\subsection{The orbit of ideal Bregman Lagrangians}
\label{Eq:TimeBregIdeal}

The ideal scaling~\eqref{Eq:IdeSca} defines a special ``one-dimensional'' subspace of {\em ideal Bregman Lagrangians}:
\begin{align}\label{Eq:BregSpacAlph}
\sL_0 = \left\{\L_{\alpha,\beta,\gamma} \, \colon \, \alpha \in \sA, \,\, \beta = 2 \alpha + \int e^\alpha, \,\, \gamma = -\alpha + \int e^{\alpha} \right\} \,\subseteq\, \sL
\end{align}
where $\sA$ is the space of all scale functions $\alpha_t \in \R$, and we use the shorthand $\int e^{\alpha} \equiv \int_{0}^t e^{\alpha_s} ds$.\footnote{More generally, we can consider the halfplane defined by $\beta \le 2 \alpha + \int e^\alpha$, for which Theorem~\ref{Thm:BregRate} holds.}
Recall by Theorem~\ref{Thm:BregRate}, the Lagrangian curve $X_t$ of an ideal Lagrangian $\L_\alpha \equiv \L_{\alpha,\beta,\gamma} \in \sL_0$ has convergence rate $\rho_t = \int e^{\alpha}$.

We observe that the family of ideal Bregman Lagrangians is closed under the action of time dilation.
Indeed, the ideal scaling~\eqref{Eq:IdeSca} holds for $\alpha,\beta,\gamma$ if and only if it holds for $\alpha^{(\tau)},\beta^{(\tau)},\gamma^{(\tau)}$. Equivalently, $\L_\alpha \in \sL_0$ if and only if $\L_{\alpha^{(\tau)}} \in \sL_0$ for any $\tau \in \sT$.

Conversely, if we start from any ideal Lagrangian $\L \in \sL_0$---say the standard Nesterov Lagrangian $\L^{\star}$ ($p=2$ in~\eqref{Eq:NestLag})
---then we can reach any other ideal Lagrangian $\L_\alpha$, $\alpha \in \sA$, by choosing the time dilation function $\tau = e^{\frac{1}{2}\int e^{\alpha}}$.\footnote{Explicitly, if $\alpha^\star_t = -\log t + \log 2$ is the scale function for $\L^\star$, then we can check that
$(\alpha^\star)^{(\tau)} = \alpha^\star(\tau) + \log \dot \tau = \alpha$.}

Therefore, we conclude that the family of ideal Bregman Lagrangians~\eqref{Eq:BregSpacAlph} is an {\em orbit} under the action of the time dilation group $\sT$. That is, we can interpret $\sL_0$ as the result of speeding up {\em any} initial Lagrangian (for concreteness $\L^\star$), over all possible time dilations:
\begin{align}\label{Eq:BregSpacAlph}
\sL_0 = \left\{(\L^\star)^{(\tau)} \, \colon \, \tau \in \sT \right\}.
\end{align}
Notice, the convergence rate transforms consistently: $\L^\star$ has rate $\rho_t = 2\log t$, so when we speed it up (by $\tau =  e^{\frac{1}{2}\int e^{\alpha}}$) to $\L_\alpha$, the rate transforms to $\rho_{\tau} = 2\log \tau = \int e^{\alpha}$, as expected. Equivalently, we can also note that as $\alpha \mapsto \alpha^{(\tau)} = (\alpha \circ \tau) + \log \dot \tau$, the rate $\int e^{\alpha} = \int e^{\alpha_t} \, dt$ transforms to $\int \dot \tau e^{(\alpha \, \circ \, \tau)} dt = \int e^{\alpha} d\tau$, consistent with the idea that we simply replace time $t (=\tid)$ by $\tau$.

\subsection{The orbit of Nesterov Lagrangians}
\label{Sec:TimeNest}

We define the subgroup of {\em polynomial time dilations} $\sTpol \subseteq \sT$:
\begin{align}\label{Eq:TimePoly}
\sTpol = \{\tau_p(t) = t^p \,\colon\, p > 0 \}
\end{align}
which is isomorphic to the multiplicative group $\R_{> 0}$.
The subgroup $\sTpol$ inherits the action on $\L_0$, but the action now partitions $\L_0$ into (sub)orbits.

We observe, the family of Nesterov Lagrangians~\eqref{Eq:NestLag} forms an orbit under the action of $\sTpol$.
Indeed by~\eqref{Eq:TimeA}, if we start from $\alpha^\star = -\log t + \log 2$ (for $\L^\star$~\eqref{Eq:NestLag}), then for any $p > 0$, the time dilation $\tau(t) = t^{\frac{p}{2}}$ sends us to $(\alpha^\star)^{(\tau)} = -\log t + \log p$.
Therefore, we can view the Nesterov flows~\eqref{Eq:NestFlow} as the result of speeding up the Lagrangian flow of $\L^\star$ ($p=2$), or any starting curve.

Moreover, as we have seen in Section~\ref{Sec:AccHigherGrad}, this speedup can be implemented in discrete time with matching convergence rate as the accelerated gradient methods $\aG_p$~\eqref{Eq:AGM}. Thus,
we can interpret the family of accelerated gradient methods $\faG$ as being the result of ``speeding up'' the algorithm $\aG_2$ (accelerated gradient descent) in discrete time---which we achieve via passage to continuous time, and at the cost of higher-order smoothness assumption on $f$.

\vspace{-4pt}
\subsection{The orbit of exp-Nesterov Lagrangians, isomorphic to Nesterov Lagrangians}
\label{Sec:TimeExpNest}
\vspace{-2pt}

From the standard Nesterov Lagrangian $\L^\star$ ($p=2$ in~\eqref{Eq:NestLag}), we can use time dilation $\tau = e^{\frac{c}{2} t}$ to reach the exp-Nesterov Lagrangian $\L = \L_c$~\eqref{Eq:expLag} with $\alpha = \log c$, which has linear rate $\rho_t = ct$, $c > 0$.
Recall, from $\L^\star$ we can generate the Nesterov Lagrangians~\eqref{Eq:NestLag} as an orbit of the action of $\sTpol$. 
The dilation $\L^\star \stackrel{\tau}{\longmapsto} \L_c$ (via the exponential time dilation $\tau = e^{\frac{c}{2}t}$) generates an equivalent orbit starting from $\L_c$, which turns out to be the family of exp-Nesterov Lagrangians~\eqref{Eq:expLag}.

Concretely, we define the subgroup of {\em linear time dilations} $\sTlin \subseteq \sT$:
\vspace{-2pt}
\begin{align}\label{Eq:TimeLin}
\sTlin = \{\theta_c(t) = ct \,\colon\, c > 0 \}
\end{align}
which is isomorphic to $\sTpol$. The discussion above says that the family of exp-Nesterov Lagrangians is an orbit under the action of $\sTlin$.
This means, in a precise sense, the Nesterov and exp-Nesterov families are isomorphic.
Moreover, we can speed up Nesterov Lagrangian to get exp-Nesterov Lagrangian via an exponential time dilation function $\tau$. 
Furthermore, the associativity of the group action means we can speed up time to go from Nesterov to exp-Nesterov and back, and the results will remain consistent. 

We can summarize our discussion by saying that all triangles in the following diagram commute (and we can reverse any arrow by replacing $\tau$ with $\tau^{-1}$):
\begin{figure}[h!]
\centering
\includegraphics[width=.98\textwidth]{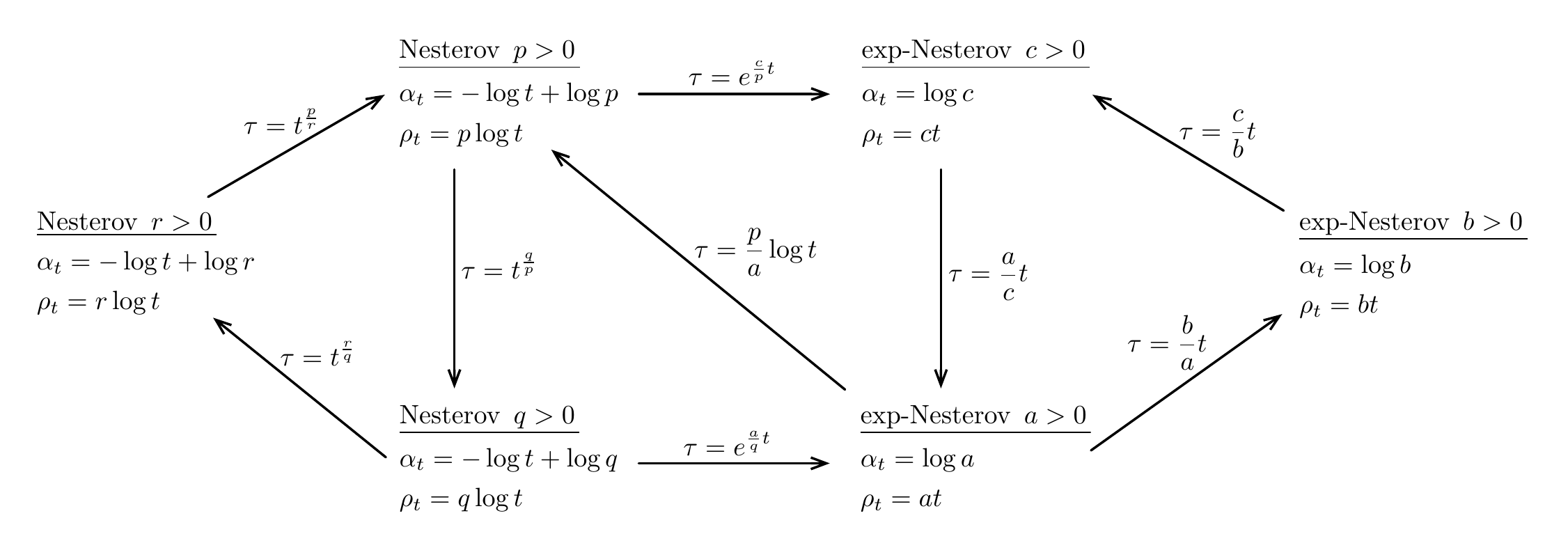}
\end{figure}

\enlargethispage{10\baselineskip}

%

\vspace{3in}


\clearpage
\appendix

\section{Natural gradient descent and mirror descent}
\label{App:MD}

We review the equivalence between natural gradient descent and mirror descent.

\subsection{Natural gradient descent and natural gradient flow}
\paragraph{Natural gradient descent.}
We can interpret natural gradient descent as the solution to a modified optimization problem, where we now measure the norm of the displacement using the Hessian metric:
\begin{subequations}\label{Eq:NatGradOpt}
\begin{align}
x_{k+1} ~&=~ x_k + v_k \notag\\
v_k & = \arg\min_{v} \left\{ f(x_k) + \langle \nabla f(x_k), v \rangle + \frac{1}{\epsilon} \cdot \frac{1}{2} \|v\|_{h(x_k)}^2 \right\}.  \label{Eq:NatGradOptb}
\end{align}
\end{subequations}
This means at each point $x$ in $\X$ we have a local inner product and norm:
\begin{align}\label{Eq:HessMet}
\langle u,v \rangle_{h(x)} = \langle u, \nabla^2 h(x) v \rangle,
\qquad
\|v\|_{h(x)}^2 = \langle v,v \rangle_{h(x)} = \langle v, \nabla^2 h(x) v \rangle.
\end{align}

\paragraph{Natural gradient flow.}
The continuous time limit ($\epsilon \to 0$ with time scaling $t = \epsilon k$) of natural gradient descent~\eqref{Eq:NatGradOpt} which we call {\em natural gradient flow} is simply gradient flow on this space:
\begin{align}\label{Eq:NatGradFlow}
\dot X_t = -\nabla^2 h(X_t)^{-1} \nabla f(X_t)
\end{align}

\paragraph{Convergence of Natural Gradient Flow}
Define the \emph{energy functional}
\[ \E_t =  t(f(X_t) - f^*) + D_h(x^*, X_t) \]
where $ x^* =\arg \min_x f(x)$ and $f^* = f(x^*)$. It has time derivative:
\[ \dot \E_t  = f(X_t) - f^* + t\langle \nabla f(X_t) ,\dot X_t \rangle - \Big\langle \frac{d}{dt} \nabla h(X_t), x^* - X_t\Big\rangle \]
Threrefore, if $X_t$ is governed by natural gradient flow~\eqref{Eq:NatGradFlow}, $\dot\E_t$ simplifies to
\begin{align}\label{Eq:MirrorFlowEdot1}
\dot \E_t = (f(X_t) - f^* + \langle  \nabla f(X_t), x^* - X_t \rangle) + t\langle \nabla f(X_t), \dot{X}_t\rangle \le t \langle \nabla f(X_t),\dot X_t\rangle \leq 0 
\end{align}
using the convexity of $f$. The energy is therefore decreasing over time: $\E_t \leq \E_0 = D(x^*, X_0)$, $t \geq 0$. Thus, we conclude that natural gradient flow~\eqref{Eq:NatGradFlow} has convergence guarantee:
\[f(X_t) - f^* \leq \frac{D_h(x^*,X_0)}{t} = O\left(\frac{1}{t}\right).\]

\subsection{Mirror descent and mirror flow}
\paragraph{Mirror descent.}
The link between natural gradient descent and mirror descent is given by the {\em Bregman divergence}. The Bregman divergence, being a first order approximation error, has the property that it approximates the Hessian norm~\eqref{Eq:HessMet} without requiring the second order derivative:
\begin{align}\label{Eq:BregmanApproxHess}
D_h(y,x) ~\approx~ \frac{1}{2} \langle y-x, \nabla^2 h(x) (y-x) \rangle = \frac{1}{2} \|y-x\|_{h(x)}^2.
\end{align}
If we replace the Hessian norm in the optimization problem~\eqref{Eq:NatGradOptb} defining natural gradient descent, then we obtain {\em mirror descent}:
\begin{align}
   x_{k+1} &= x_k + v_k \notag\\
   v_k &=  \arg\min_{v} \left\{ f(x_k) + \langle \nabla f(x_k), v \rangle + \frac{1}{\epsilon} \,D_h(x_k + v, x_k) \right\}. \label{Eq:MirrDescOpt2}
\end{align}
Setting the derivative of~\eqref{Eq:MirrDescOpt2} to zero, we can also write mirror descent explicitly as:
\begin{align}\label{Eq:MirrDesc}
\nabla h(x_{k+1}) = \nabla h(x_k) - \epsilon \nabla f(x_k).
\end{align}

\paragraph{Mirror flow.}
The continuous time limit ($\epsilon \to 0$ with time scaling $t = \epsilon k$) of mirror descent~\eqref{Eq:MirrDesc} is the {\em mirror flow}, which is the system:
\begin{subequations}\label{Eq:MirrFlow}
\begin{align}
Z_t &= \nabla h(X_t)   \label{Eq:MirrFlowa} \\
\dot Z_t &= -\nabla f(X_t)   \label{Eq:MirrFlowb}
\end{align}
\end{subequations}
Therefore, mirror flow~\eqref{Eq:MirrFlow} is equivalent to natural gradient flow~\eqref{Eq:NatGradFlow}:
\begin{align}\label{Eq:MirrNatGrad}
\frac{d}{dt}\nabla h(X_t) =  \dot Z_t = -\nabla f(X_t)
\qquad\Leftrightarrow\qquad
\dot X_t = -\nabla^2 h(X_t)^{-1} \nabla f(X_t).
\end{align}
Furthermore, mirror flow still has the same $O(1/t)$ convergence in continuous time for any convex function $f$. While mirror descent and mirror flow have matching convergence rates, it is difficult to prove any convergence rate for natural gradient descent; all we can say is that it is a descent method if $f$ is smooth with respect to $\|\cdot\|_{h(x)}$.

\subsection{Natural gradient flow and mirror flow equivalence}

In~\eqref{Eq:MirrNatGrad} we showed an ``algebraic trick'' that shows how the same differential equation can be written in two different ways, demonstrating the equivalence between mirror flow and (natural) gradient flow. Formally, we can understand this trick as the manifestation of the following property: {\em 
mirror flow is the pushforward of natural gradient flow under the mapping $\Phi = \nabla h$}. In particular, mirror flow is also a gradient flow in the ``dual manifold'' $\mathcal{Z} = \nabla h(\X)$.
To illustrate this property, we show how the gradient flow changes when we transform the space.

\paragraph{Mapping the space.}
Suppose we map $\X$ to $\Z = \Phi(\X)$ by a bijective smooth map:
$$\Phi \colon \X \to \Z$$
with inverse map (also smooth) $\Psi = \Phi^{-1} \colon \Z \to \X$. The objective function $f \colon \X \to \R$ transforms to a new objective function $\tilde f \colon \Z \to \R$ given by:
$$\tilde f = f \circ \Psi$$
so that if $z = \Phi(x)$:
\begin{align}\label{Eq:FuncPush}
\tilde f(z) = \tilde f(\Phi(x)) = f(x).
\end{align}
However, note that $\tilde f$ is not necessarily convex (in $z$), even if $f$ is (in $x$).

\paragraph{How the gradient changes.}
We will use the following notation:
\begin{align*}
\partial_x f(x_0,z_0) = \frac{\partial f(x,z)}{\partial x} \Big|_{(x,z) = (x_0,z_0)}
\end{align*}
The new objective function $\tilde f$ now has gradient at $z_0$ (in the new space $\Z$):
\begin{align}\label{Eq:GradPush}
\nabla \tilde f(z_0) = \partial_z \tilde f(z_0) = \partial_z (f \circ \Psi)(z_0) = J_{\Psi}(z_0) \, \partial_x f(\Psi(z_0)) = J_\Psi(z_0)  \, \nabla f(\Psi(z_0))
\end{align}
where $J_\Psi(z_0)$ is the Jacobian (partial derivatives) of $\Psi(z)$ at $z = z_0$, represented as a matrix:
\begin{align}
\big( J_\Psi(z_0) \big)_{ij} = \partial_i \big(\Psi(z_0)\big)_j = \frac{\partial \Psi_j(z)}{\partial z_i} \Big|_{z = z_0}.
\end{align}

\paragraph{How the metric changes.}
Suppose $\X$ has metric $\gm(x)$ at point $x$ (e.g., Hessian metric $\gm = \nabla^2 h$). Then we can obtain a corresponding metric in $\Z = \Phi(\X)$ which is the {\em pullback metric} of $\gm$ under the inverse mapping $\Psi = \Phi^{-1}$:
\begin{align*}
\tilde{\gm} = \Psi^* \gm.
\end{align*}
Explicitly, this means the inner product at the point $z_0 = \Phi(x_0)$ is given by:
\begin{align*}
\langle u,v \rangle_{z_0}
 ~=~ \langle J_\Psi(z_0) u, \, J_\Psi(z_0) v \rangle_{x_0}
 ~=~ \langle J_\Psi(z_0) u, \gm(x_0) J_\Psi(z_0) v \rangle.
\end{align*}
That is, the metric $\gm(x_0)$ at $x_0 = \Psi(z_0)$ now becomes the metric $\tilde{\gm}(z_0)$ at $z_0$:
\begin{align}\label{Eq:PullbackMet}
\tilde{\gm}(z_0)
 ~=~ (\Psi^* \gm)(z_0)
 ~=~ J_\Psi(z_0)^\top \, (\gm \circ \Psi)(z_0) \, J_\Psi(z_0).
 \end{align}

\paragraph{How the gradient flow changes.}
Suppose in the original space $\X$ we have gradient flow, which (with the general definition of metric $\gm$) is given by:
\begin{align*}
\dot X_t = -\gm(X_t)^{-1} \nabla f(X_t).
\end{align*}
In the new space $\Z = \Phi(\X)$ with the pushforward metric $\tilde\gm = \Psi^* \gm$ and the new objective function $\tilde f = f \circ \Psi$, the natural gradient flow equation~\eqref{Eq:NatGradFlow} becomes:
\begin{align*}
\dot Z_t = -\tilde\gm(Z_t)^{-1} \nabla \tilde f(Z_t).
\end{align*}
Plugging in~\eqref{Eq:PullbackMet} then gives us:
\begin{align}
\dot Z_t &= - \big[ J_\Psi(Z_t)^\top \, \gm(\Psi(Z_t)) \, J_\Psi(Z_t) \big]^{-1} \: J_\Psi(Z_t) \, \nabla f(\Psi(Z_t)) \notag \\
&= - J_\Psi(Z_t)^{-1} \, \gm(\Psi(Z_t))^{-1} \, \nabla f(\Psi(Z_t)). \label{Eq:PushGradFlow}
\end{align}

\paragraph{Mirror map from Legendre duality.}
Now consider the Hessian metric again:
$$\gm = \nabla^2 h.$$
In this case, there is a very nice choice of $\Psi$, called the {\em mirror map}:
$$\Psi = \nabla h^*.$$
Here $h^* \colon \X^* \to \X$ is the {\em Legendre dual function}, defined on the space of linear functionals $\X^*$:
\begin{align}\label{Eq:LegDual}
h^*(z) = \sup_{x} \;  \langle z,x \rangle - h(x).
\end{align}
The optimum in~\eqref{Eq:LegDual} is achieved by $x$ satisfying $z = \nabla h(x)$. So for all $x \in \X$, we have the relation:
\begin{align}\label{Eq:LegDualx}
h^*(\nabla h(x)) = \langle \nabla h(x), x \rangle - h(x).
\end{align}
Similarly, since $(h^*)^* = h$, for all $z \in \Z$ we also have:
\begin{align}\label{Eq:LegDualz}
h(\nabla h^*(z)) = \langle \nabla h^*(z), z \rangle - h^*(z).
\end{align}
Comparing~\eqref{Eq:LegDualx} and~\eqref{Eq:LegDualz}, we conclude:
\begin{align*}
z = \nabla h(x)
\qquad\Leftrightarrow\qquad
x = \nabla h^*(z)
\end{align*}
which means $\nabla h^* = (\nabla h)^{-1}$, so for all $z \in \Z$:
$$\nabla h(\nabla h^*(z)) = z.$$
Differentiating (calculating the Jacobian with respect to $z$) of the expression above gives us:
$$\nabla^2 h^*(z) \: \nabla^2 h(\nabla h^*(z)) = I.$$
So with the Hessian metric $\gm = \nabla^2 h$ and the mirror map $\Psi = \nabla h^*$ ($\Phi = \nabla h$), we have:
\begin{align*}
J_\Psi(Z_t) &= \partial_z \nabla h^*(Z_t) = \nabla^2 h^*(Z_t) \\
\gm(\Psi(Z_t)) &= \nabla^2 h(\nabla h^*(Z_t)) = \left(\nabla^2 h^*(Z_t) \right)^{-1}
\end{align*}
Notice how the choice of the mirror map makes $J_\Psi(Z_t)$ and $\gm(\Psi(Z_t))$ cancel each other. Therefore, the pushforward of the natural gradient flow~\eqref{Eq:PushGradFlow} in $\Z$ is indeed the mirror flow~\eqref{Eq:MirrFlow}:
\begin{align*}
\dot Z_t = - \nabla f(X_t)
\end{align*}

\subsection{Natural gradient flow as massless limit of Hessian Lagrangian flow}
\label{Sec:NaturalMasslessLim}
We consider the damped Lagrangian using the Hessian metric:
\begin{equation}\label{Eq:HessLagrangian}
\L(X_t,\dot X_t, t) = e^{\gamma_t}\Big(\frac{m}{2}||\dot X_t||_{h(X_t)}^2 - f(X_t)\Big) 
\end{equation}
The Euler-Lagrange equation corresponding to \eqref{Eq:HessLagrangian} is:
\begin{equation}\label{Eq:ELHessLag}
\frac{m}{2}\nabla^3h(X_t)\dot X_t \dot X_t + \nabla^2h(X_t)(m\ddot X_t + m \dot \gamma_t \dot X_t) + \nabla f(X_t) = 0
\end{equation}
when $\gamma_t = t/m$, then in the limit $m\rightarrow 0$, the equation \eqref{Eq:ELHessLag} converges to the first order equation
\[ \nabla^2 h(X_t)\dot X_t +  \nabla f(X_t) = 0 \]
which is equivalent to the natural gradient flow \eqref{Eq:NatGradFlow}

\section{Deferred proofs}

\subsection{Proof of Lemma~\ref{Lem:Resd}}
\label{App:ConvHigherGrad}

We prove Lemma~\ref{Lem:Resd}, following the technique of~\cite[Theorem~1]{Nesterov08}.
Since $f$ is $\frac{(p-1)!}{\epsilon}$-smooth of order $p-1$, and we define $x_{k+1}$ by~\eqref{Eq:HigherGrad}, we have the following bound:
\begin{align}\label{Eq:DiscGradBound1}
f(x_{k+1})
  \:\le\: \min_{x} \left\{ f(x) + \frac{2}{\epsilon} \cdot \frac{1}{p} \|x-x_k\|^p \right\}.
\end{align}
Moreover, choosing $v = x^* - x_k$ in~\eqref{Eq:DiscGradBound1} gives us the bound:
\begin{align*}
f(x_{k+1}) - f^* \le \frac{2}{\epsilon} \cdot \frac{1}{p} \|x^* - x_k\|^p.
\end{align*}
For any $\lambda \in (0,1)$, consider the midpoint:
$$x_\lambda = x^* + (1-\lambda)(x_k-x^*) = \lambda x^* + (1-\lambda) x_k.$$
Then from the convexity of $f$~\eqref{Eq:Conv}, we have the bound $f(x_\lambda) \le \lambda f^* + (1-\lambda) f(x_k)$.
Plugging this to~\eqref{Eq:DiscGradBound1} with the choice $v = x_\lambda - x_k$, we find:
\begin{align*}
f(x_{k+1})
\:&\le\: f(x_\lambda) + \frac{2}{\epsilon} \cdot \frac{1}{p} \|x_\lambda-x_k\|^p
\:\le\: \lambda f^* + (1-\lambda) f(x_k) + \frac{2}{\epsilon} \cdot \frac{1}{p} R^p \, \lambda^p. 
\end{align*}
Then with our notation $\delta_k = f(x_k) - f^*$, we can write the last inequality above more precisely as:
\begin{align*}
\delta_{k+1} \:\le\: (1-\lambda) \delta_k + \frac{2}{\epsilon} \cdot \frac{1}{p} R^p \, \lambda^p.
\end{align*}
Minimizing the right hand side with respect to $\lambda$ yields the optimal bound:
\begin{align}\label{Eq:DeltaBound1}
\delta_{k+1} \le \delta_k - \frac{(p-1)}{p} \cdot \left(\frac{\epsilon \delta_k^p}{2 R^p}\right)^{\frac{1}{p-1}}.
\end{align}
Now consider the energy functional $e_k = \delta_k^{-\frac{1}{p-1}}$; we have:
\begin{align}
e_{k+1} - e_k
  \:=\: \frac{1}{ \delta_{k+1}^{\frac{1}{p-1}}} - \frac{1}{ \delta_k^{\frac{1}{p-1}}} 
  \:&=\:  \frac{\delta_{k}^{\frac{1}{p-1}}-\delta_{k+1}^{\frac{1}{p-1}}}{\delta_{k+1}^{\frac{1}{p-1}} \cdot \delta_{k}^{\frac{1}{p-1}}}  
  \:=\: \frac{\delta_k-\delta_{k+1}}{\delta_{k+1}^{\frac{1}{p-1}} \cdot \delta_{k}^{\frac{1}{p-1}}} \cdot \frac{1}{\big( \sum_{i=0}^{p-2} \delta_k^{\frac{i}{p-1}} \cdot \delta_{k+1}^{\frac{p-2-i}{p-1}} \big)}. \label{Eq:E3Bound1}
\end{align}
We know that $\G_p$ is a descent method, so 
$e_{k+1} \ge e_k$. The summation in the denominator of~\eqref{Eq:E3Bound1} can be bounded above by $(p-1) \delta_k^{\frac{p-2}{p-1}}$, and we can lower bound~$\delta_k-\delta_{k-1}$ using~\eqref{Eq:DeltaBound1}. Therefore: 
\begin{align}\label{Eq:E3Bound2}
e_{k+1}-e_k
  \:\ge\: \frac{(p-1)}{p} \cdot \left(\frac{\epsilon \delta_k^{p}}{2R^{p}}\right)^{\frac{1}{p-1}} \cdot \frac{1}{\delta_k^{\frac{2}{p-1}}} \cdot \frac{1}{(p-1) \delta_k^{\frac{p-2}{p-1}}}
  \:=\: \frac{1}{p} \cdot \left(\frac{\epsilon}{2R^p}\right)^{\frac{1}{p-1}}.
\end{align}
Summing~\eqref{Eq:E3Bound2} and telescoping the terms, we get:
\begin{align*}
e_k \:\ge\: e_k-e_0 \:\ge\: \frac{k}{p} \cdot \left(\frac{\epsilon}{2R^p}\right)^{\frac{1}{p-1}}
\end{align*}
which gives us the conclusion. \qed

\subsection{Proof of Theorem~\ref{Thm:RescaledGradFlow}}
\label{App:RescaledGrad}
Consider the energy functional \eqref{Eq:RescGradFlowEnergy}:
\begin{align}\label{Eq:E2}
E_t = \big(f(X_t)-f^*\big)^{-\frac{1}{p-1}}.
\end{align}
Its time derivative is:
\begin{align*}
\dot E_t
 \:&=\: -\frac{1}{(p-1)} \cdot \frac{\langle \nabla f(X_t), \dot X_t \rangle}{(f(X_t)-f^*)^{\frac{p}{p-1}}} 
\end{align*}
If $X_t$ evolves following the rescaled gradient flow~\eqref{Eq:RescGradFlow2}, then $\dot E_t$ simplifies to:
\begin{align}\label{Eq:E2dot}
\dot E_t
  \:&=\: \frac{1}{(p-1)} \cdot \left(\frac{\|\nabla f(X_t)\|_*}{f(X_t)-f^*}\right)^{\frac{p}{p-1}}.
\end{align}
Notice that by the convexity of $f$, we have
\begin{align*}
0 \le f(X_t)-f^* \le \langle \nabla f(X_t), X_t-x^* \rangle \le \|\nabla f(X_t)\|_* \cdot \|X_t-x^*\|
\end{align*}
and therefore, from~\eqref{Eq:E2dot} we obtain a bound:
\begin{align*}
\dot E_t
  \:\geq\: \frac{1}{(p-1)} \cdot \frac{1}{\|X_t-x^*\|^{\frac{p}{p-1}}}.
\end{align*}
Integrating, this gives a lower bound on $E_t$:
\begin{align}\label{Eq:E2dot1}
E_t \:\geq\: E_0 + \frac{1}{(p-1)} \int_0^t \frac{1}{\|X_\tau-x^*\|^{\frac{p}{p-1}}} d\tau.
\end{align}
Furthermore, under the additional mild assumption that the level sets are bounded, the bound~\eqref{Eq:E2dot1} simplifies to:
\begin{align*}
E_t \:\geq\: E_0 + \frac{t}{(p-1) \, R^{\frac{p}{p-1}}}.
\end{align*}
Now recalling the definition~\eqref{Eq:E2}, this lower bound becomes an upper bound on the function values:
\begin{align*}
f(X_t) - f^* \:\le\: \left(E_0 + \frac{t}{(p-1) R^{\frac{p}{p-1}}} \right)^{-(p-1)}.
\end{align*}
Finally, replacing $E_0 \ge 0$ by $0$ completes the proof. \qed

\subsection{Proof of Lemma~\ref{Lem:HigherGrad}}
\label{App:LemHigherGrad}

We follow the proof of~\cite[Lemma~6]{Nesterov08}. Since $y_k$ solves the optimization problem~\eqref{Eq:AGMy}, it satisfies the optimality condition:
\begin{align}\label{Eq:HigherGradLem1}
\sum_{i=1}^{p-1} \frac{1}{(i-1)!} \nabla^i f(x_k) \, (y_k-x_k)^{i-1} + \frac{2}{\epsilon} \|y_k-x_k\|^{p-2} \, (y_k-x_k) = 0.
\end{align}
Furthermore, since $\nabla^{p-1} f$ is $\frac{(p-1)!}{\epsilon}$-Lipschitz, we have the following error bound on the $(p-2)$-nd order Taylor expansion of $\nabla f$ :
\begin{align}\label{Eq:HigherGradLem2}
\left\|\nabla f(y_k) - \sum_{i=1}^{p-1} \frac{1}{(i-1)!} \nabla^i f(x_k) \, (y_k-x_k)^{i-1}\right\|_* \le \frac{1}{\epsilon} \|y_k-x_k\|^{p-1}.
\end{align}
Substituting~\eqref{Eq:HigherGradLem1} to the square of~\eqref{Eq:HigherGradLem2} and writing $r = \|y_k-x_k\|$, we obtain:
\begin{align}
\frac{r^{2p-2}}{\epsilon^2}
\,\ge\, \left\|\nabla f(y_k) + \frac{2r^{p-2}}{\epsilon} \, (y_k-x_k)\right\|_*^2. \notag
\end{align}
Upon expanding the square and rearranging, we get the inequality:
\begin{align}\label{Eq:HigherGradLem3}
\langle \nabla f(y_k), x_k-y_k \rangle
\,\ge\, \frac{\epsilon}{4r^{p-2}} \|\nabla f(y_k)\|_*^2 + \frac{3r^p}{4\epsilon}.
\end{align}
Note that if $p=2$, then the first term in~\eqref{Eq:HigherGradLem3} above already implies the desired bound~\eqref{Eq:HigherGradLem}. Now assume $p \ge 3$. The right hand side of~\eqref{Eq:HigherGradLem3} is a convex function of $r$, and it is minimized by $r^* = 
\big\{\frac{p-2}{3p} \, \epsilon^2 \|\nabla f(y_k)\|_*^2\big\}^{\frac{1}{2p-2}}$, yielding a lower bound of~\eqref{Eq:HigherGradLem3} that is now independent of $r$:
\begin{align}
\langle \nabla f(y_k), x_k-y_k \rangle
\,&\ge\, \frac{1}{4} \left(\left(\frac{3p}{p-2}\right)^{\frac{p-2}{2p-2}} + \left(\frac{p-2}{3p}\right)^{\frac{p}{2p-2}} \right) 
\epsilon^{\frac{1}{p-1}} \|\nabla f(y_k)\|_*^{\frac{p}{p-1}} \notag \\
&\ge\, \frac{1}{4} \, \epsilon^{\frac{1}{p-1}} \|\nabla f(y_k)\|_*^{\frac{p}{p-1}}  \notag
\end{align}
as desired. \qed

\clearpage
\bibliography{acceleration_arxiv.bbl}

\end{document}